\documentclass[oneside,english]{amsbook}
\usepackage[T1]{fontenc}
\usepackage[latin9]{inputenc}
\usepackage{geometry}
\usepackage{color}
\usepackage{verbatim}
\usepackage{textcomp}
\usepackage{amstext}
\usepackage{amsthm}
\usepackage{amssymb}
\usepackage{graphicx}
\usepackage[all]{xy}

\makeatletter

\newcommand{\noun}[1]{\textsc{#1}}

\numberwithin{section}{chapter}
\numberwithin{equation}{section}
\numberwithin{figure}{section}
\theoremstyle{plain}
\newtheorem{thm}{\protect\theoremname}
  \theoremstyle{plain}
  \newtheorem{prop}[thm]{\protect\propositionname}
  \theoremstyle{remark}
  \newtheorem*{rem*}{\protect\remarkname}
  \theoremstyle{definition}
  \newtheorem{defn}[thm]{\protect\definitionname}
  \theoremstyle{plain}
  \newtheorem{cor}[thm]{\protect\corollaryname}
  \theoremstyle{definition}
  \newtheorem{example}[thm]{\protect\examplename}
  \theoremstyle{plain}
  \newtheorem*{thm*}{\protect\theoremname}
  \theoremstyle{plain}
  \newtheorem{lem}[thm]{\protect\lemmaname}

\makeatother

\usepackage{babel}
  \providecommand{\corollaryname}{Corollary}
  \providecommand{\definitionname}{Definition}
  \providecommand{\examplename}{Example}
  \providecommand{\lemmaname}{Lemma}
  \providecommand{\propositionname}{Proposition}
  \providecommand{\remarkname}{Remark}
  \providecommand{\theoremname}{Theorem}
\providecommand{\theoremname}{Theorem}

\begin{document}

\title[Oscillation bounds for amenable ergodic averages]{Fluctuation bounds for ergodic averages of amenable groups on uniformly
convex Banach spaces}
\begin{abstract}
We study fluctuations of ergodic averages generated by actions of
amenable groups. In the setting of an abstract ergodic theorem for
locally compact second countable amenable groups acting on uniformly
convex Banach spaces, we deduce a highly uniform bound on the number
of fluctuations of the ergodic average for a class of Følner sequences
satisfying an analogue of Lindenstrauss's temperedness condition.
Equivalently, we deduce a uniform bound on the number of fluctuations
over long distances for arbitrary Følner sequences. As a corollary,
these results imply associated bounds for a continuous action of an
amenable group on a $\sigma$-finite $L^{p}$ space with $p\in(1,\infty)$. 
\end{abstract}

\date{January 22 2018}

\author{\noun{Andrew Warren}}

\author{Carnegie Mellon University}

\maketitle
\tableofcontents{}

\chapter*{Introduction}

The topic at hand is indicated by the following diagram:
\[
\xymatrix{\mbox{Classical Ergodic Theory}\ar[d]\ar[r] & \mbox{Ergodic Theory of Group Actions}\ar@{-->}[d]\\
\mbox{Effective Ergodic Theory}\ar@{-->}[r] & \mbox{Effective Ergodic Theory of Group Actions}
}
\]

Classical ergodic theory concerns itself with the study of measure-preserving
transformations $T:X\rightarrow X$ on probability spaces $(X,\mu)$.
In the earliest applications, $X$ was typically understood to be
the phase space of some physical system, and $T$ encoded discrete
time evolution of the system; $\mu$ would be some natural measure
on the phase space which was invariant under time evolution. The most
basic results in classical ergodic theory are the \emph{ergodic theorems}
of von Neumann and Birkhoff, which respectively assert that if $f:X\rightarrow\mathbb{R}$
is some observable feature of the system, and we consider the \emph{average
value }of the observable after a certain amount of time, namely $\frac{1}{N}\sum_{i=0}^{N-1}f\circ T^{i}$,
then (i) if $f\in L^{2}(X,\mu)$ then $\frac{1}{N}\sum_{i=0}^{N-1}f\circ T^{i}$
converges in $L^{2}$ norm, and (ii) if $f\in L^{1}(X,\mu)$, then
$\frac{1}{N}\sum_{i=0}^{N-1}f\circ T^{i}$ converges in $L^{1}$ and
pointwise $\mu$-almost surely. 

Typically, classical ergodic theory is understood as a type of \emph{soft
analysis} \textemdash{} convergence theorems are stated in asymptotic
form without explicit constants, and proofs are carried out using
abstract non-computational tools from functional analysis. This is
not entirely a coincidence: very early in the development of the theory,
it became apparent that it is often \emph{impossible} to find explicit
rates of convergence in ergodic theorems as they are usually stated.
The domain of \emph{effective ergodic theory} seeks to determine when
results in ergodic theory can be made computationally explicit, perhaps
under restricted circumstances. A notable feature of the area is that
numerous statements in ergodic theory can be naturally recast in terms
of \emph{weak modes of uniform convergence }originally developed within
constructive mathematics and proof theory. 

Likewise, work in classical ergodic theory eventually determined both
that (i) for many ergodic theorems, the choice of underlying space
$X$ was not of central importance, and (ii) numerouse theorems could
be just as easily stated (and less easily, proved) if, relaxing the
metaphor of time evolution, one considers multiple transformations
acting concurrently on a space, or even an entire group of transformations
acting on a space. The \emph{ergodic theory of group actions }seeks
to understand how the choice of acting group alters the character
of the theory. 

This thesis offers a contribution to the effective ergodic theory
of group actions, and in particular to the mean ergodic theorem for
actions of groups which are \emph{amenable}. The amenable groups comprise
a large and varied class, and include a number of families of groups
of independent interest, such as: all locally compact abelian groups,
upper triangular matrix groups, solvable groups, and others. The ergodic
theorems for amenable groups may also be understood as the most general
extension of the original, classical ergodic theorems in terms of
modifying the acting group, such that the resulting generalization
actually still contains the classical theorem as a special case. 

In this document, we assume that the reader has some degree of comfort
with classical ergodic theory. However, no background in effective
ergodic theory or the ergodic theory of amenable group actions is
assumed. In Chapter 1, we discuss weak modes of uniform convergence
and their relevance for classical ergodic theory. In Chapter 2, we
give a survey of some aspects of the theory of amenable groups, in
order to give a flavour for the field, and discuss how both the proof
of the mean ergodic theorem and the proof that the mean ergodic theorem
has no uniform rate of convergence generalize to the amenable setting.
Finally, in Chapter 3, we show that the mean ergodic theorem for amenable
group actions has effective convergence information in terms of an
explicit uniform bound on fluctuations of the ergodic average over
long distances. 

The commutative diagram above may also be used as a leitfaden: Chapters
1 and 2 may be read out of order, but parts of both are needed for
work in Chapter 3. 

Finally, we should remark that, although there is no direct use of
tools from mathematical logic in this work, nonetheless this research
has been influenced in many ways by the logical research programme
of \emph{proof mining}. A very brief discussion of two connections
between this work and the proof mining literature appears in Appendix
B. 

This work was completed as part of the author's Master's thesis. 
There are a number of people whose help in the course of this research
has proved invaluable. I would especially like to thank my advisor,
Jeremy Avigad, for helpful suggestions too numerous to mention; Clinton
Conley, for introducing me to amenable groups; Yves Cornulier and
Henry Towsner, for helpful discussions when the project was in its
early stages; Máté Szabó, for pleasant distractions; and Theodore
Teichman, for unflagging moral support. Naturally, all remaining errors
are my own. 
\begin{flushright}
Andrew Warren
\par\end{flushright}

\begin{flushright}
Pittsburgh, Pennsylvania
\par\end{flushright}

\begin{flushright}
January 2019
\par\end{flushright}

\chapter{Beyond Rates of Convergence}

In this chapter, we review rates of convergence and other forms of
convergence information which have proved relevant in classical ergodic
theory. In Section 1, we introduce weaker forms of uniform convergence
than a uniform rate of convergence, and discuss some connections with
computability theory and constructive mathematics. In Section 2, we
give a well-known proof that there is no uniform rate of convergence
in the von Neumann and Birkhoff ergodic theorems, and review some
related work on explicit convergence information in these ergodic
theorems. 

\section{Modes of Uniform Convergence}

To say that every sequence in some class $\mathcal{S}$ converges
is, for many purposes, too vague. In what follows, we consider several
distinct ways that the sequences in $\mathcal{S}$ might all converge
in a \emph{uniform} way.
\begin{enumerate}
\item There is a \emph{uniform rate of convergence}: there exists a function
$r:\mathbb{R}^{+}\rightarrow\mathbb{N}$ such that for every $\varepsilon>0$,
and every $m,n\geq r(\varepsilon)$, it holds for all $(x_{n})\in S$
that $\Vert x_{n}-x_{m}\Vert<\varepsilon$.
\item There is a \emph{uniform fluctuation bound}: there exists a function
$\lambda:\mathbb{R}^{+}\rightarrow\mathbb{N}$ such that for every
$\varepsilon>0$, and every $(x_{n})\in\mathcal{S}$, the \emph{number
of $\varepsilon$-fluctuations} is at most $\lambda(\varepsilon)$.
That is to say, for every $(x_{n})$ and every finite sequence $n_{1},\ldots,n_{k}$
such that for all $i\in[1,k)$, $\Vert x_{n_{i}}-x_{n_{i+1}}\Vert\geq\varepsilon$,
it necessarily holds that $k\leq\lambda(\varepsilon)$.
\item There is a \emph{uniform fluctuation bound at distance $\beta$}:
a weakened version of the previous form of uniform convergence, which
will be especially important for us. A bound on fluctuations at distance
$\beta$ only checks for fluctuations which are ``far enough apart''.
Explicitly, for each $\varepsilon>0$, let $\beta(-,\varepsilon):\mathbb{N}\rightarrow\mathbb{N}$
be some strictly increasing function. Then there is a uniform fluctuation
bound at distance $\beta$ provided that there exists some function
$\lambda_{\beta}:\mathbb{R}^{+}\rightarrow\mathbb{N}$ such that for
every $\varepsilon>0$, every $(x_{n})$, and every finite sequence
$n_{1},\ldots,n_{k}$ \emph{with the property that $\beta(n_{i},\varepsilon)\leq n_{i+1}$
for all $i\in[0,k)$ }such that for all $i\in[1,k)$, $\Vert x_{n_{i}}-x_{n_{i+1}}\Vert\geq\varepsilon$,
it necessarily holds that $k\leq\lambda_{\beta}(\varepsilon)$. (Setting
$\beta(n,\varepsilon)=n+1$ for each $\varepsilon>0$ reduces this
condition to (2) above.)
\item There is a \emph{uniform rate of metastability}: given $\varepsilon>0$,
there exists a functional $\Phi(-,\varepsilon):\mathbb{N}^{\mathbb{N}}\rightarrow\mathbb{N}$
such that for all strictly increasing $F:\mathbb{N}\rightarrow\mathbb{N}$,
it holds for all $(x_{n})$ that there exists an $N\leq\Phi(F,\varepsilon)$
such that for all $n,m\in[N,F(N)]$, $\Vert x_{n}-x_{m}\Vert<\varepsilon$.
In other words, if we are searching for a finitary period of stability
for the sequence $(x_{n})$ of length specified by $F$, then $\Phi$
gives an upper bound on how far we have to search.
\end{enumerate}
We have listed these forms of uniformity in descending order of strength.
\begin{prop}
If $\mathcal{S}$ is some family of sequences, then with respect to
the preceding list of statements, $1\Rightarrow2\Rightarrow3\Rightarrow4$.
In general, none of the converse implications hold.
\end{prop}

\begin{proof}
$(1\Rightarrow2)$ If $r(\varepsilon)$ is a uniform rate of convergence
for $\mathcal{S}$, this means that all possible indices $n_{i},n_{i+1}$
with $n_{i}<n_{i+1}$ and with the property that $\Vert x_{i}-x_{i+1}\Vert\geq\varepsilon$
must have $i<r(\varepsilon)$. Therefore the number of $\varepsilon$-fluctuations
is at most $r(\varepsilon)$, and we can just set $\lambda(\varepsilon)=r(\varepsilon)$.

$(2\Rightarrow3)$ Obvious.

$(3\Rightarrow4)$ Fix an $\varepsilon>0$. Define $\tilde{F}(n):=\max\{F(n),\beta(n,\varepsilon/2)\}$.
Also define an increasing sequence of naturals by $N_{1}=1$ and $N_{k+1}=\tilde{F}(N_{k})$.
Now, observe that if there is an $\varepsilon$-fluctuation in the
first of the intervals $[N_{i-1},\tilde{F}(N_{i-1})]$, $[N_{i},\tilde{F}(N_{i})]$,
$[N_{i+1},\tilde{F}(N_{i+1})]$, then (thanks to the triangle inequality)
it must be possible to pick an index $j_{1}$ from $[N_{i-1},\tilde{F}(N_{i-1})]$
and an index $j_{2}$ from $[N_{i+1},\tilde{F}(N_{i+1})]$ such that
$\Vert x_{j_{1}}-x_{j_{2}}\Vert\geq\varepsilon/2$, and this $\varepsilon/2$-fluctuation
is at distance $\beta(-,\varepsilon/2)$ (since $j_{2}\geq\max\{F(N_{i}),\beta(N_{i},\varepsilon/2)\}\geq\beta(N_{i},\varepsilon/2)\geq\beta(j_{1},\varepsilon/2)$).
Therefore, if a sequence $(x_{n})$ has a $\varepsilon$-fluctuation
in every interval $[N_{i},\tilde{F}(N_{i})]$ for $i=1,\ldots,2\lambda_{\beta}(\varepsilon/2)+3$,
then we can find at least $\lambda_{\beta}(\varepsilon/2)+1$ many
$\varepsilon/2$-fluctuations at distance $\beta$ in $(x_{n})$.
Consequently, if we assume that $\mathcal{S}$ has $\lambda_{\beta}(\varepsilon/2)$
as a uniform bound on the number of $\varepsilon/2$-fluctuations
at distance $\beta(-,\varepsilon/2)$, then at least one interval
$[N_{i},\tilde{F}(N_{i})]$ (for $i\in[1,2\lambda_{\beta}(\varepsilon)+3]$)
must \emph{not} have an $\varepsilon$-fluctuation. This implies that
for every $(x_{n})\in\mathcal{S}$, we can pick an $N\leq N_{2\lambda_{\beta}(\varepsilon/2)+3}$
so that $[N,\tilde{F}(N)]$ (and therefore $[N,F(N)]$) has no $\varepsilon$-fluctuations.
In other words there is a uniform bound on the rate of metastability
of the form $\Phi(F,\varepsilon)=\tilde{F}^{2\lambda_{\beta}(\varepsilon/2)+3}(1)$,
where the exponent $2\lambda_{\beta}(\varepsilon/2)+3$ denotes iterated
application of $\tilde{F}$. 

$(4\not\Rightarrow3)$ 
First, we actually prove that $4\not\Rightarrow2$. Since 2 is a special
case of 3, this tells us that there is at least one distance function
$\beta$ for which a bound on the rate of metastability does not give
a bound on the number of $\varepsilon$-fluctuations at distance $\beta$.
However, this does not show that $4$ is strictly weaker than $3$
for \emph{every} distance function $\beta$. Hence this proves $4\not\Rightarrow3$
but only in a weak sense. We will then modify the proof that $4\not\Rightarrow2$
to get a proof that given an \emph{arbitrary} $\beta$, a bound on
the rate of metastability is strictly weaker than a bound on the number
of $\varepsilon$-fluctuations at distance $\beta$, thus proving
$4\not\Rightarrow3$ in a stronger sense as well. 

We borrow a counterexample from Avigad and Rute \cite{avigad2015oscillation}.
Let $\mathcal{S}$ denote a countable family of binary sequences where
the $j$th sequence is identically zero for the first $j-1$ terms,
and then oscillates $j$ times between $0$ and $1$ beginning at
the $j$th element, and then is constant thereafter. Evidently $\mathcal{S}$
has no uniform bound on the number of $\varepsilon$-fluctuations
for any $\varepsilon\leq1$. Nonetheless it has a uniform bound on
the rate of metastability. Indeed, fix an $\varepsilon\in(0,1)$ and
take any increasing function $F:\mathbb{N}\rightarrow\mathbb{N}$.
Pick some $(x_{n})\in\mathcal{S}$. If $F(1)<j$, then $[1,F(1)]$
has no fluctuations. Otherwise, $F(1)\geq j$; in this case, observe
that at least one of the intervals 
\[
[F(1),F^{2}(1)],[F^{2}(1),F^{3}(1)],\ldots,[F^{F(1)+1}(1),F^{F(1)+2}(1)]
\]
has no fluctuations (in particular, the last interval in this list!),
simply because $F$ is increasing, and therefore the fact that $F(1)\geq j$
implies that $F^{F(1)+1}\geq2j$. Consequently, 
\[
\Phi(F,\varepsilon)=\begin{cases}
1 & \varepsilon>1\\
F^{F(1)+1} & \varepsilon\leq1
\end{cases}
\]
is a uniform bound on the rate of metastability for $\mathcal{S}$. 

Now we adapt this argument to show that $4\not\Rightarrow3$ for \emph{arbitrary}
distance function $\beta$. Using $\mathcal{S}$ from above, we define
a new family $\mathcal{S}^{\prime}$ in the following way. Given $(x_{n})\in\mathcal{S}$,
define $(y_{n})\in\mathcal{S}^{\prime}$ by 
\[
y_{i}=\begin{cases}
x_{1} & 1\leq i<\tilde{\beta}(1)\\
x_{2} & \tilde{\beta}(1)\leq i<\tilde{\beta}^{2}(1)\\
 & \vdots\\
x_{n} & \tilde{\beta}^{n-1}(1)\leq i<\tilde{\beta}^{n}(1)\\
 & \vdots
\end{cases}
\]
where $\tilde{\beta}(n)$ is shorthand for $\beta(n,1)$. It follows
that the number of $1$-fluctuations of distance $\beta$ in $(y_{n})$
is the same as the number of $1$-fluctuations in $(x_{n})$: given
a sequence of indices $n_{1}<n_{2}<\ldots<n_{k}$ witnessing the $1$-fluctuations,
we can pick a term $y_{j_{1}}$ of $(y_{n})$ in the block of terms
which are all equal to $x_{n_{1}}$, and then there will be a term
of $(y_{n})$ equal to $x_{n_{1}+1}$ which is at distance at least
$\beta(j_{1},1)$ from $y_{j_{1}}$, so \emph{a fortiori} we can find
a term equal to $x_{n_{2}}$ which is at distance $\beta(j_{1},1)$
from $y_{j_{1}}$. And so on. (This uses the fact that $\tilde{\beta}$
is increasing \textemdash{} for $k<\tilde{\beta}^{n-1}(1)$ we have
that $\tilde{\beta}(k)<\tilde{\beta}^{n}(1)$.) In this way we can
find $k$ $\varepsilon$-fluctuations at distance $\beta$ in $(y_{n})$.
So for our particular original family of sequences $\mathcal{S}$,
the $j$th sequence will be constant zero in the interval $[1,\tilde{\beta}^{j-1}(1))$,
and then alternate between one and zero for the next $j$-many intervals
of the form $[\tilde{\beta}^{i-1}(1),\tilde{\beta}^{i}(1))$ for $i\in[j,2j]$. 

Now, let $F:\mathbb{N}\rightarrow\mathbb{N}$ be an increasing function.
Consider some $(y_{n})\in\mathcal{S}^{\prime}$. If $F(1)<\tilde{\beta}^{j-1}(1)$
then $[1,F(1)]$ contains no fluctuations. Otherwise $F(1)\geq\tilde{\beta}^{j-1}(1)$.
Now, if it so happens that $F^{2}(1)<\tilde{\beta}^{j}(1)$, then
$[F(1),F^{2}(1)]$ is a sub-interval of $[\tilde{\beta}^{j-1}(1),\tilde{\beta}^{j}(1))$,
and hence has no fluctuations. Otherwise $F^{2}(1)\geq\tilde{\beta}^{j}(1)$.
We can now ask whether $F^{3}(1)<\tilde{\beta}^{j+1}(1)$. 

Repeat this case-wise reasoning over and over, until we reach $F^{j+2}(1)$.
If we make it this far without finding an interval with no fluctuations,
this means that $F^{j+2}(1)\geq\tilde{\beta}^{2j}(1)$. However, after
$\tilde{\beta}^{2j}(1)$, we know by construction that $(y_{n})$
is constant forever. Hence $[F^{j+2}(1),\infty)$ has no fluctuations.
So in particular, neither does $[F^{F(1)+2}(1),F^{F(1)+3}(1)]$, since
(if we made it this far in the case reasoning) we know that $F(1)\geq j$. 

In other words, at least one of the intervals 
\[
[1,F(1)],[F(1),F^{2}(1)],\ldots,[F^{F(1)+2}(1),F^{F(1)+3}(1)]
\]
has no fluctuations. It follows that 
\[
\Phi(F,\varepsilon)=\begin{cases}
1 & \varepsilon>1\\
F^{F(1)+2} & \varepsilon\leq1
\end{cases}
\]
is a uniform bound on the rate of metastability for $\mathcal{S}^{\prime}$. 

$(3\not\Rightarrow2)$ Fix an $\varepsilon>0$. If $\beta(n,\varepsilon)$
is dominated by $n+k_{\varepsilon}$ for some constant $k_{\varepsilon}$,
then in fact 3 \emph{does} imply 2: one can simply take $\lambda(\varepsilon)=2k_{\varepsilon}\lambda_{\beta}(\varepsilon)$.
\footnote{To see this: take any sequence which is already at distance $\beta$.
Then in between indices $n_{i}$ and $n_{i+1}$, there are $2k_{\varepsilon}$-many
``forbidden'' indices which might add some $\varepsilon$-fluctuations.
Since for \emph{every} sequence of indices we can get at most $2k_{\varepsilon}$
times as many $\varepsilon$-fluctuations by relaxing the ``at distance
$\beta$'' restriction, the same holds for the maximum number of
$\varepsilon$-fluctuations.} Suppose, therefore, that $\beta(n,\varepsilon)$ is \emph{superaffine}:
namely, that there is an increasing sequence $(n_{k})$ such that
for every $n_{k}$, $\beta(n_{k},\varepsilon)\geq n_{k}+k$. Now consider
the following family of binary sequences: each sequence is $0$ everywhere,
except that whenever $k$ is even, the $k$th sequence has a series
of $k$-many oscillations between $0$ and 1 immediately following
the $n_{k}$th index. Then, for every $\varepsilon\leq1$, it holds
that $\lambda(\varepsilon)=2$ is a uniform upper bound on the number
of $\varepsilon$-fluctuations at distance $\beta$, but there is
no uniform upper bound on the number of $\varepsilon$-fluctuations. 

$(2\not\Rightarrow1)$ Consider a family of binary sequences such
that for the $n$th sequence, the first $n$ terms are all $0$ and
the remaining terms are all 1. This family of sequences has a uniform
bound on fluctuations for every $\varepsilon$ but for $\varepsilon\leq1$
there is no uniform rate of convergence. 
\end{proof}
\begin{rem*}
The preceding proposition is not the end of the story on distinct
modes of uniform convergence. See for instance the recent paper of
Towsner \cite{towsner2017nonstandard}, which gives an infinite hierarchy
of distinct modes of uniform convergence in between a uniform bound
on fluctuations and a uniform rate of metastability. 
\end{rem*}
Rather than considering families of sequences and modes of uniform
convergence, we can also ask whether a single convergent sequence
has, for instance, a rate of convergence which is \emph{computable}.
Notably, in this setting the situation is almost identical: if a sequence
has a computable rate of convergence, then it also has a computable
number of $\varepsilon$-fluctuations, which implies a computable
number of $\varepsilon$-fluctuations at (computable) distance $\beta$,
which in turn implies a computable rate of metastability. In fact,
in this direction, all of the \emph{proofs} are nearly identical!
For observe that in the proofs of the forward directions of the preceding
proposition, at each stage we defined a new modulus in terms of the
previous one \textemdash{} for instance, defining a rate of metastability
in terms of a bound on the number of $\varepsilon$-fluctuations at
distance $\beta$. At each stage, our new definition was simple enough
that the new modulus is \emph{relatively }computable in terms of the
previous one, so if the previous modulus is assumed to be computable
then we're done.

Just as in the case of modes of uniform convergence, the converse
implications are all false. However, the proofs \textemdash{} showing,
for instance, that a computable bound on the number of $\varepsilon$-fluctuations
does not imply a computable rate of convergence, for a single sequence
\textemdash{} are somewhat different in flavour than the converse
directions in the previous proposition. For further discussion in
this vein, we refer the reader to §5 of the paper by Avigad and Rute
\cite{avigad2015oscillation}, and §4 of the paper by Kohlenbach and
Safarik \cite{kohlenbach2014fluctuations} (but see also Appendix
B.1).

In any case, there is an extremely strong analogy between distinct
modes of uniform convergence, and distinct modes of computable convergence.
For this reason, work on weak modes of uniform convergence often draws
on developments from computable analysis/constructive mathematics.
Yet another perspective on the preceding proposition is the constructivist
one: in constructive mathematics, it is not meaningful to assert that
a sequence converges without giving more explicit information about
how this convergence occurs. A very frequent occurrence in constructive
mathematics is that classical notions ``bifurcate'' into multiple
inequivalent constructive analogues; what our discussion indicates
is that the classical notion of convergence has many inequivalent
constructive analogues, including but not limited to ``this sequence
has an explicit rate of convergence'', ``this sequence has an explicit
bound on the number of $\varepsilon$-fluctuations'', etc. 

All this is to say that the results presented later in this document
can be interpreted as giving a more \emph{uniform} version of existing
ergodic theorems, as well as giving a version of existing ergodic
theorems which is sufficiently \emph{computationally explicit} to
be \emph{constructively admissible}. 

\section{Convergence Issues in Classical Ergodic Theory}

We begin this section by reviewing an important negative result concerning
uniform convergence in ergodic theory.
\begin{thm}
\label{thm:Krengel}Let $(X,\mu)$ be a probability space, and let
$T:X\rightarrow X$ be an invertible ergodic measure-preserving transformation
on $(X,\mu)$. Let $p\in[1,\infty)$. Then there is no uniform rate
of convergence for the class of ergodic averages $\left\{ \frac{1}{N}\sum_{i=0}^{N-1}f\circ T^{i};f\in L^{p}(X,\mu)\right\} $,
either in $L^{p}$ norm or pointwise almost surely. 
\end{thm}

\begin{proof}
We follow the argument indicated by Krengel (who remarks that this
result was already a well-known folk theorem). The strategy will be
to produce a sequence of measurable subsets $(E_{n})$ of $X$, such
that 

\begin{enumerate}
\item $\mu(E_{n})=\frac{1}{2}$ for every $n$, so in particular (as $N\rightarrow\infty$)
$\frac{1}{N}\sum_{i=0}^{N-1}\mathbf{1}_{E_{n}}\circ T^{i}$ converges
to the constant function $\frac{1}{2}$ in $L^{p}$ and pointwise
a.s., 
\item $E_{n+1}$ is produced by modifying $E_{n}$ on a set of small measure
(say less than $2^{-n}$), so that asymptotically $(E_{n})$ converges
to some measurable set $E\subset X$ which also has measure $\frac{1}{2}$,
and thus (by the mean and pointwise ergodic theorems) $\frac{1}{N}\sum_{i=0}^{N-1}\mathbf{1}_{E}\circ T^{i}$
converges to the constant function $\frac{1}{2}$ in $L^{p}$ and
pointwise a.s., and 
\item the sequence $\frac{1}{N}\sum_{i=0}^{N-1}\mathbf{1}_{E}\circ T^{i}$
converges to $\frac{1}{2}$ (in $L^{p}$ and pointwise a.s.) more
slowly than some prespecified rate of convergence. 
\end{enumerate}
To make that last point more precise, we first fix a sequence $(\alpha_{N})$
of positive reals which converges monotonically to zero. We will then
follow the construction outlined above to produce a measurable set
$E$ with $\mu(E)=\frac{1}{2}$, such that $\limsup_{N\rightarrow\infty}\alpha_{N}^{-1}\Vert\frac{1}{N}\sum_{i=0}^{N-1}\mathbf{1}_{E}\circ T^{i}-\frac{1}{2}\Vert_{p}=\infty$,
and moreover, for almost all $x\in X$, $\limsup_{N\rightarrow\infty}\alpha_{N}^{-1}\vert\frac{1}{N}\sum_{i=0}^{N-1}\mathbf{1}_{E}\circ T^{i}(x)-\frac{1}{2}\vert=\infty$.

In what follows, use the standard shorthand $A_{N}f:=\frac{1}{N}\sum_{i=0}^{N-1}f\circ T^{i}$. 

To initialise the construction, let $E_{1}$ be any subset of $X$
with measure $\frac{1}{2}$. Define also $N_{0}=0$, $N_{1}=1$. Now
suppose that we have already constructed $E_{n}$, again with measure
$\frac{1}{2}$, and have defined $N_{n}>N_{n-1}$, as well as $M_{n-1}>N_{n-1}$.

Now let $\varepsilon_{n}\in(0,\min\{\alpha_{M_{n-1}}/2^{n},1/(N_{n}2^{n})\}]$.
Let $p_{n}$ be an integer which is sufficiently large that $p_{n}^{-1}<\varepsilon_{n}/4$.
Let $M_{n}$ be an integer such that $M_{n}>N_{n}$, and $16\alpha_{M_{n}}n<p_{n}^{-1}$,
and such that there exists some $K_{n}>N_{n}$ such that $M_{n}=4K_{n}$.
Now define $N_{n+1}=p_{n}M_{n}$. Since $T$ is ergodic, and therefore
a.s. aperiodic, we can invoke the Rokhlin tower lemma and produce
a measurable set $B_{n}$ such that the sets $T^{-k}B_{n}$ are all
disjoint for $k\in[0,N_{n+1})$, and $\bigsqcup_{k=0}^{N_{n+1}-1}T^{-k}B_{n}$
has measure at least $1-\varepsilon_{n}/4$. 

Using $B_{n}$, we define the set 
\[
C_{n}:=\bigsqcup_{k=0}^{2M_{n}-1}T^{-k}B_{n}
\]
and note that (simply because $T$ is measure-preserving) $(1-\varepsilon_{n}/4)2M_{n}/N_{n+1}\leq\mu(C_{n})\leq2M_{n}/N_{n+1}$;
by definition of $N_{n+1}$, this reduces to 
\[
(1-\varepsilon_{n}/4)2p_{n}^{-1}\leq\mu(C_{n})\leq2p_{n}^{-1}
\]
so that in particular 
\[
(1-\varepsilon_{n}/4)16n\alpha_{M_{n}}<\mu(C_{n})<\varepsilon_{n}/2.
\]

We now define $E_{n+1}$ by modifying $E_{n}$ on $C_{n}$ in the
following manner. For each $x\in B_{n}$, let $v_{n}(x)$ be the number
of indices in $[0,2M_{n})$ for which $T^{-i}x\in E_{n}$. Define
\[
B_{n,k}:=\{x\in B_{n}\mid v_{n}(x)=k\}.
\]
Note that $\bigsqcup_{i=0}^{2M_{n}-1}\left(E_{n}\cap T^{-i}B_{n,k}\right)=k\cdot\mu(B_{n,k})$.
(Why?) Now, for each $k\in[0,2M_{n}]$, we first remove $\bigsqcup_{i=0}^{2M_{n}-1}\left(E_{n}\cap T^{-i}B_{n,k}\right)$
from $E_{n}$, and then replace it with a set of equal measure in
the following way:

\begin{enumerate}
\item If $k\geq M_{n}$, add $T^{-i}B_{n,k}$ to $E_{n}$ for every $0\leq i<k$. 
\item If $k<M_{n}$, add $T^{-i}B_{n,k}$ to $E_{n}$ for every $2M_{n}-k\leq i<2M_{n}.$
\end{enumerate}
In either case, we've added in $k$ many disjoint sets, each of which
have the same measure as $\mu(B_{n,k})$. Consequently, at the end
of each stage $k$, the measure of (the modified version of) $E_{n}$
is unchanged. 

After repeating this procedure for all $0<k<2M_{n}$, we declare the
resulting set to be $E_{n+1}$. Notably, $E_{n+1}$ has the property
that for every $x\in B_{n}$, the set $\{T^{-i}x\mid0\leq i<M_{n}\}$
is either entirely contained in $E_{n+1}$, or entirely in $E_{n+1}$.
Moreover, $\mu(E_{n+1})=\frac{1}{2}$, and $\mu(E_{n}\Delta E_{n+1})\leq\mu(C_{n})$.

Now, let us first consider norm convergence. Define 
\[
D_{n}:=\bigsqcup_{k=M_{n}}^{5K_{n}-1}T^{-k}B_{n}.
\]
Note that $\mu(D_{n})=K_{n}\mu(B_{n})$. Likewise, notice that notice
that for any $z\in D_{n}$, $z=T^{-k}x$ for some $x\in B_{n}$, $k\in[M_{n},5K_{n})$,
and thus if $i\in(K_{n},M_{n}]$, then $T^{i}z=T^{i-k}x$ and $k-i\in[0,M_{n})$.
It follows that either $T^{i}z\in E_{n+1}$ for every $i\in(K_{n},M_{n}]$,
or $T^{i}x\notin E_{n+1}$ for every $i\in K_{n},M_{n}]$. Since $M_{n}=4K_{n}$,
this implies that either $A_{M_{n}}\mathbf{1}_{E_{n+1}}(x)\geq3/4$,
or $A_{M_{n}}\mathbf{1}_{E_{n+1}}(x)\leq1/4$. Regardless, it follows
that on $D_{n}$, $\vert A_{M_{n}}\mathbf{1}_{E_{n+1}}-\frac{1}{2}\vert\geq\frac{1}{4}$.
This implies that 
\[
\left\Vert A_{M_{n}}\mathbf{1}_{E_{n+1}}-\frac{1}{2}\right\Vert _{1}\geq\frac{1}{4}\cdot\mu(D_{n})
\]
so therefore 
\begin{align*}
\alpha_{M_{n}}^{-1}\left\Vert A_{M_{n}}\mathbf{1}_{E_{n+1}}-\frac{1}{2}\right\Vert _{1} & \geq\alpha_{M_{n}}^{-1}\frac{1}{4}\cdot\mu(D_{n})\\
 & =\alpha_{M_{n}}^{-1}\frac{1}{4}K_{n}\mu(B_{n})\\
 & \geq\alpha_{M_{n}}^{-1}\frac{1}{16}M_{n}\frac{1}{N_{n+1}}(1-\varepsilon_{n}/4)\\
 & =\alpha_{M_{n}}^{-1}\frac{1}{16}p_{n}^{-1}(1-\varepsilon_{n}/4)\\
 & >\alpha_{M_{n}}^{-1}\frac{1}{16}(16\alpha_{M_{n}}n)(1-\varepsilon_{n}/4)\\
 & =n(1-\varepsilon_{n}/4).
\end{align*}

Now, note that 
\begin{align*}
\left\Vert A_{M_{n}}\mathbf{1}_{E_{n+1}}-\frac{1}{2}\right\Vert _{1} & \leq\left\Vert A_{M_{n}}\mathbf{1}_{E}-\frac{1}{2}\right\Vert _{1}+\left\Vert A_{M_{n}}\mathbf{1}_{E_{n+1}}-A_{M_{n}}\mathbf{1}_{E}\right\Vert _{1}\\
 & \leq\left\Vert A_{M_{n}}\mathbf{1}_{E}-\frac{1}{2}\right\Vert _{1}+\left\Vert A_{M_{n}}\right\Vert \cdot\left\Vert \mathbf{1}_{E_{n+1}}-\mathbf{1}_{E}\right\Vert _{1}
\end{align*}
and also (using the dominated convergence theorem) 
\[
\left\Vert \mathbf{1}_{E_{n+1}}-\mathbf{1}_{E}\right\Vert _{1}=\mu(E_{n+1}\Delta E)\leq\sum_{i=n+1}^{\infty}\mu(E_{i}\Delta E_{i+1})\leq\sum_{i=n+1}^{\infty}\mu(C_{i}).
\]
And since $\mu(C_{i})\leq2M_{i}/N_{i+1}=2p_{i}^{-1}$, and we have
that $p_{i}^{-1}<\varepsilon_{i}/4\leq(\alpha_{M_{i-1}}2^{-i})/4$,
and since $(\alpha_{n})$ is decreasing, we can compute 
\[
\sum_{i=n+1}^{\infty}\mu(C_{i})=2\sum_{i=n+1}^{\infty}p_{i}^{-1}<\frac{1}{2}\sum_{i=n+1}^{\infty}\varepsilon_{i}\leq\frac{1}{2}\sum_{i=n+1}^{\infty}\alpha_{M_{i-1}}2^{-i}\leq\frac{1}{2^{n+1}}\alpha_{M_{n}}
\]
 so therefore (since $\Vert A_{M_{n}}\Vert=1$) 
\[
\left\Vert A_{M_{n}}\mathbf{1}_{E_{n+1}}-\frac{1}{2}\right\Vert _{1}<\left\Vert A_{M_{n}}\mathbf{1}_{E}-\frac{1}{2}\right\Vert _{1}+\frac{1}{2^{n+1}}\alpha_{M_{n}}
\]
which, together with our lower bound on $\alpha_{M_{n}}^{-1}\left\Vert A_{M_{n}}\mathbf{1}_{E_{n+1}}-\frac{1}{2}\right\Vert _{1}$,
implies that 
\[
\alpha_{M_{n}}^{-1}\left\Vert A_{M_{n}}\mathbf{1}_{E}-\frac{1}{2}\right\Vert _{1}>n(1-\varepsilon_{n}/4)-\frac{1}{2^{n+1}}.
\]
Evidently this implies that 
\[
\limsup_{N\rightarrow\infty}\alpha_{N}^{-1}\left\Vert A_{N}\mathbf{1}_{E}-\frac{1}{2}\right\Vert _{1}=\infty
\]
as desired. (Since $\Vert\cdot\Vert_{p}\geq\Vert\cdot\Vert_{1}$,
this actually suffices for all $p\in[1,\infty)$.) 

The strategy for pointwise a.s. convergence is similar. Fix some $L\in[2M_{n},N_{n+1})$,
and put $N^{\prime}=L+1$ and $N^{\prime\prime}=L-M_{n}+1$. Consider
some point $\eta\in T^{-L}B_{n}$. Let $\ell^{\prime}$ denote the
number of indices $k\in[0,N^{\prime})$ such that $T^{k}\eta\in A_{n+1}$,
and likewise let $\ell^{\prime\prime}$ denote the number of indices
$k\in[0,N^{\prime\prime})$ such that $T^{k}\eta\in A_{n+1}$. Note
that by our construction, for every $\eta\in T^{-L}B_{n}$, it holds
either that $T^{k}\eta\in E_{n+1}$ for all $k\in(L-M_{n},L]$, or
that $T^{k}\eta\notin E_{n+1}$ for all $k\in(L-M_{n},L]$ (since
in this regime, $T^{k}\eta=T^{-i}x$ for some $x\in B_{n}$, $i\in[0,M_{n})$).
It follows, therefore, that either $\ell^{\prime}=\ell^{\prime\prime}$,
or $\ell^{\prime}=\ell^{\prime\prime}+M_{n}$. 

Note that $A_{N^{\prime}}\mathbf{1}_{En+1}(\eta)=\ell^{\prime}/N^{\prime}$,
and similarly $A_{N^{\prime\prime}}\mathbf{1}_{A_{n+1}}(\eta)=\ell^{\prime\prime}/N^{\prime\prime}$.
First, suppose that either $\ell^{\prime\prime}/N^{\prime\prime}\geq3/4$
or $\ell^{\prime\prime}/N^{\prime\prime}\leq1/4$. In either case,
we have that 
\[
\alpha_{N^{\prime\prime}}^{-1}\left\vert A_{N^{\prime\prime}}\mathbf{1}_{E_{n+1}}(\eta)-\frac{1}{2}\right\vert \geq\alpha_{M_{n}}^{-1}\left\vert \ell^{\prime\prime}/N^{\prime\prime}-\frac{1}{2}\right\vert >16np_{n}\cdot\frac{1}{4}>4n
\]
where we have used the fact that $N^{\prime\prime}>M_{n}$ and $(\alpha_{n})$
is decreasing. 

Otherwise, $\ell^{\prime\prime}/N^{\prime\prime}\in(1/4,3/4)$. If
$\ell^{\prime}=\ell^{\prime\prime}+M_{n}$, then 
\begin{align*}
\frac{\ell^{\prime}}{N^{\prime}}-\frac{\ell^{\prime\prime}}{N^{\prime\prime}} & =\frac{\ell^{\prime\prime}+M_{n}}{N^{\prime}}-\frac{\ell^{\prime\prime}}{N^{\prime\prime}}\\
 & =\frac{(\ell^{\prime\prime}+M_{n})N^{\prime\prime}-\ell^{\prime\prime}(N^{\prime\prime}+M_{n})}{N^{\prime}N^{\prime\prime}}\\
 & =\frac{M_{n}}{N^{\prime}}\cdot\frac{N^{\prime\prime}-\ell^{\prime\prime}}{N^{\prime\prime}}\\
\mbox{(since }N^{\prime}\leq N_{n+1}) & >\frac{M_{n}}{N_{n+1}}\cdot\frac{1}{4}\\
 & =p_{n}^{-1}\cdot\frac{1}{4}\\
 & >4n\alpha_{M_{n}}.
\end{align*}
Whereas if $\ell^{\prime}=\ell^{\prime\prime}$, then 
\begin{align*}
\frac{\ell^{\prime}}{N^{\prime}}-\frac{\ell^{\prime\prime}}{N^{\prime\prime}} & =\frac{\ell^{\prime\prime}N^{\prime\prime}-\ell^{\prime\prime}(N^{\prime\prime}+M_{n})}{(N^{\prime})(N^{\prime\prime})}=\frac{\ell^{\prime\prime}}{N^{\prime\prime}}\cdot\frac{M_{n}}{N^{\prime}}\geq\frac{1}{4}\cdot\frac{M_{n}}{N_{n+1}}=\frac{1}{4}p_{n}^{-1}>4n\alpha_{M_{n}}.
\end{align*}
Therefore, in either case, we have that 
\[
\alpha_{M_{n}}^{-1}\left\vert A_{N^{\prime}}\mathbf{1}_{E_{n+1}}-A_{N^{\prime\prime}}\mathbf{1}_{E_{n+1}}\right|>4n
\]
so by the triangle inequality, 
\[
\max\left\{ \alpha_{M_{n}}^{-1}\left\vert A_{N^{\prime}}\mathbf{1}_{E_{n+1}}-\frac{1}{2}\right|,\alpha_{M_{n}}^{-1}\left\vert A_{N^{\prime\prime}}\mathbf{1}_{E_{n+1}}-\frac{1}{2}\right|\right\} >2n
\]
and therefore, since $(\alpha_{n})$ is decreasing, that 
\[
\max\left\{ \alpha_{N^{\prime}}^{-1}\left\vert A_{N^{\prime}}\mathbf{1}_{E_{n+1}}-\frac{1}{2}\right|,\alpha_{N^{\prime\prime}}^{-1}\left\vert A_{N^{\prime\prime}}\mathbf{1}_{E_{n+1}}-\frac{1}{2}\right|\right\} >2n.
\]

Thus, regardless of the value of $\ell^{\prime\prime}/N^{\prime\prime}$,
we have that 
\[
\sup_{M_{n}<N\leq N_{n+1}}\alpha_{N}^{-1}\left\vert A_{N}\mathbf{1}_{E_{n+1}}-\frac{1}{2}\right|>2n.
\]

Now, quantifying over all $L\in[2M_{n},N_{n+1})$, we see that the
preceding argument is valid on the set $\bigsqcup_{L=2M_{n}}^{N_{n+1}-1}T^{-L}B_{n}$.
This corresponds to the entire Rokhlin tower except the initial segment
$\bigsqcup_{i=0}^{2M_{n}-1}T^{-i}B_{n}$, and therefore has measure
at least 
\begin{align*}
1-\varepsilon_{n}/4-2M_{n}\mu(B_{n}) & \geq1-\varepsilon_{n}/4-(2M_{n}/N_{n+1})\\
 & =1-\varepsilon_{n}/4-2p_{n}^{-1}\\
 & >1-3\varepsilon_{n}/4.
\end{align*}
In order to replace $\mathbf{1}_{E_{n+1}}$ with $\mathbf{1}_{E}$,
let's estimate $\mu(E_{n+1}\Delta E)$ a second time, this time using
the other upper bound on $\varepsilon_{n}$: 
\[
\mu(E_{n+1}\Delta E)\leq\sum_{i=n+1}^{\infty}\mu(C_{i})<\frac{1}{2}\sum_{i=n+1}^{\infty}\varepsilon_{i}\leq\frac{1}{2}\sum_{i=n+1}^{\infty}\frac{1}{N_{i}2^{i}}<\frac{1}{2^{n+1}N_{n+1}}
\]
(where we have merely used the fact that $N_{i}<N_{i+1}$).

It follows, therefore, that 
\[
\sup_{M_{n}<N\leq N_{n+1}}\alpha_{N}^{-1}\left\vert A_{N}\mathbf{1}_{E}-\frac{1}{2}\right|>2n
\]
holds on a set of measure at least $1-3\varepsilon_{n}/4-2^{-n}$.
Obviously this error term goes to zero as $n$ goes to $\infty$;
thus, we finally conclude that 
\[
\limsup_{N\rightarrow\infty}\alpha_{N}^{-1}\left\vert A_{N}\mathbf{1}_{E}-\frac{1}{2}\right|=\infty\quad\mbox{almost surely.}
\]
\end{proof}
\begin{rem*}
It is also possible to deduce the lack of a rate of convergence in
this setting from a more abstract argument which exploits the fact
that $\mathbb{Z}$ is amenable. In fact we will do this in Theorem
\ref{thm:Amenable-rate}.

One might ask what extra assumptions are needed to get a rate of convergence
for ergodic averages. What the previous proof shows is that it does
not suffice put stronger assumptions than ergodicity on the transformation
(say, that $T$ is mixing). Rather, the problem stems from the fact
that the class of $L^{2}$ functions is ``too big'' \textemdash{}
if, for instance, we work on a subspace of $L^{2}$ which does not
contain indicator functions (!) then the preceding proof breaks down.
For instance, it is possible, for a number of special dynamical systems
\cite{Chernov:2008}, to prove a uniform \emph{exponential decay of
correlations} (which in turn gives a rate of convergence of ergodic
averages) for the class of $L^{2}$ functions which are Hölder continuous
and have upper bounded Hölder seminorm $K$, for some prespecified
constant $K$.
\end{rem*}
An analogous negative result also holds concerning whether the rate
of convergence of a single ergodic average is computable.
\begin{thm}
There exists a measurable, computable subset $E$ of $[0,1]$ and
a computable measure-preserving transformation $T$ on $[0,1]$ such
that there is no computable bound on the rate of convergence of $A_{n}\mathbf{1}_{E}$,
either in $L^{p}$ or pointwise almost surely. 
\end{thm}

\begin{proof}
See Theorem 5.1 of the paper by Avigad et al. \cite{avigad2010local}.
\end{proof}
However, it is worth clarifying a potential point of confusion regarding
the statement of Theorem \ref{thm:Krengel}. The proof of this result
indicates that, given a specific rate of convergence $(\alpha_{n})$,
it is possible to find a function $f$ (in fact an indicator function)
such that $\int f=\frac{1}{2}$, and $(A_{n}f)$ converges to $\frac{1}{2}$
at a rate even slower than $(\alpha_{n})$. This does \emph{not} necessarily
mean that, having selected this $f$ it is impossible to compute the
rate of convergence of $(A_{n}f)$. In fact, this rate of convergence
is computable (given $f$ and $T$) whenever we also know the norm
of the limit of $(A_{n}f)$, so in particular whenever $T$ is ergodic
and $f$ is a function with known integral, as in the proof of Theorem
\ref{thm:Krengel}. 
\begin{thm}
Let $T$ be a nonexpansive operator on a separable Hilbert space (e.g.
a Koopman operator on $L^{2}(X,\mu)$ with $(X,\mu)$ separable).
Let $f^{*}$ denote the limit of $(A_{n}f)$. Then a bound on the
rate of convergence of $(A_{n}f)$, can be computed from $f$, $T$,
and $\Vert f^{*}\Vert$. 
\end{thm}

\begin{proof}
This is Theorem 5.2 of Avigad et al. \cite{avigad2010local}, but
we sketch the argument. Given a vector $f$ in the Hilbert space,
we know that $f=f^{*}+g$ where $f^{*}$ is the projection of $f$
onto the $T$-invariant subspace. Likewise, it is possible to approximate
$g$ with the sequence $(g_{i})$, where $g_{i}$ is the projection
of $f$ onto the subspace spanned by $\{f-Tf,Tf-T^{2}f,\ldots,T^{i}f-T^{i+1}f\}$.
It follows that $g_{i}\rightarrow g$ and moreover $\Vert g_{i}\Vert$
is nondecreasing. 

A short computation (Lemma 2.5 in Avigad et al.) shows that $\Vert g-g_{i}\Vert\leq\sqrt{2(\Vert g\Vert-\Vert g_{i}\Vert)\Vert f\Vert}$.
In turn, $\Vert g\Vert^{2}=\Vert f\Vert^{2}-\Vert f^{*}\Vert^{2}$
simply from orthogonality. Lastly, it can be shown (see discussion
preceding Lemma 2.3 in Avigad et al.) that $g_{i}$ can be written
in the form $u_{i}-Tu_{i}$, where $u_{i}$ is given explicitly in
terms of $g_{i}$, $T$, and $f$. Therefore, using the telescoping
estimate 
\[
A_{n}(u-Tu)=\frac{1}{n}(u-T^{n}u);\quad\Vert A_{n}(u-Tu)\Vert\leq\frac{2}{n}\Vert u\Vert
\]
and the estimate 
\begin{align*}
\Vert A_{n}f-A_{m}f\Vert & =\Vert A_{n}g-A_{m}g\Vert\\
 & \leq\Vert A_{n}g_{i}-A_{m}g_{i}\Vert+\Vert A_{n}(g-g_{i})\Vert+\Vert A_{m}(g-g_{i})\Vert\\
 & \leq\Vert A_{n}g_{i}\Vert+\Vert A_{m}g_{i}\Vert+2\Vert g-g_{i}\Vert
\end{align*}
we can then bound the rate of convergence of $(A_{n}f)$ in the following
manner. First, search for the least $i$ such that $\Vert g-g_{i}\Vert<\varepsilon/4$.
Then, compute the $u_{i}$ associated to this $g_{i}$, and compute
$\Vert u_{i}\Vert$. Pick $m$ large enough that $2\Vert u_{i}\Vert/m<\varepsilon/4$
(and thus $2\Vert u_{i}\Vert/n<\varepsilon/4$ for all $n\geq m$).
It follows that for all $m\geq n$, $\Vert A_{n}f-A_{m}f\Vert<\varepsilon$. 
\end{proof}
\begin{rem*}
Obviously the preceding result is an example of an \emph{effective}
convergence theorem in ergodic theory, but not a \emph{uniform} one. 
\end{rem*}
It is the absence of a uniform rate of convergence for the von Neumann
and Birkhoff ergodic theorems that has motivated the investigation
of weaker forms of uniform convergence, including bounds on the number
of fluctuations and bounds on the rate of metastability. Let us briefly
mention some existing results in this direction.

In Avigad et al. \cite{avigad2010local}, the authors give an explicit
bound on the rate of metastability for $(A_{n}f)$ which depends only
on $\Vert f\Vert/\varepsilon$ in the setting of an action of a single
nonexpansive transformation on a Hilbert space. A short but inexplicit
proof using ultraproduct methods was subsequently given by Avigad
and Iovino \cite{avigad2013ultraproducts}. More generally, the result
of Avigad et al. was subsequently generalized to uniformly convex
Banach spaces by Kohlenbach and Leu\c{s}tean \cite{kohlenbach2009quantitative}.
In turn, this result was strengthened by Avigad and Rute \cite{avigad2015oscillation},
who gave an explicit bound on the number of fluctuations for $(A_{n}f)$,
with $T$ a nonexpansive operator on a uniformly convex Banach space. 

Before discussing existing results in the pointwise a.s. setting,
let us mention the relationship between bounds on the number of fluctuations
and \emph{upcrossing inequalities}. Given an interval $(\alpha,\beta)$
in $\mathbb{R}$, and a real sequence $(x_{n})$, an \emph{upcrossing
}of $(\alpha,\beta)$ corresponds to a pair of indices $n_{i}<n_{i+1}$
such that $x_{n_{i}}\leq\alpha$ and $x_{n_{i+1}}\geq\beta$. We can
then consider finite subsequences of $(x_{n})$ such that for every
odd $i$, $x_{n_{i}}\leq\alpha$ and $x_{n_{i+1}}\geq\beta$. Then
the maximum length of such a finite subsequence, divided by two, gives
the \emph{number of upcrossings} of the interval $(\alpha,\beta)$
in the sequence $(x_{n})$. For a single sequence $(x_{n})$, we can
ask whether $(x_{n})$ has an explicit/computable upper bound on the
number of upcrossings of some interval $(\alpha,\beta)$; for a family
of sequences, we can ask whether the family has a uniform upper bound
on the number of upcrossings of $(\alpha,\beta)$. If so, the result
is known as an upcrossing inequality. (It is also possible to define
\emph{downcrossings} in the same fashion.)

It is clear that if a sequence has at most $k$ $\varepsilon$-fluctuations,
then for every interval $(\alpha,\beta)$ with $\beta-\alpha\geq\varepsilon$,
there can be at most $k/2$ upcrossings of $(\alpha,\beta)$. Conversely,
if we know that a sequence is \emph{bounded} in some interval $[a,b]$,
it is possible do deduce a bound on the number of $\varepsilon$-fluctuations
by, for instance, partitioning $[a,b]$ into sub-intervals $(\alpha_{i},\beta_{i})$
such that $\beta_{i}-\alpha_{i}<\varepsilon/2$, and observing that
every $\varepsilon$-fluctuation must be either an upcrossing or a
downcrossing with respect to some sub-interval. 

The theorem in analysis which most famously has a natural statement
in terms of an upcrossing inequality is of course the martingale convergence
theorem, which can be stated in the following form. Let $(M_{n})$
be a real-valued martingale adapted to some filtrated probability
space $(\Omega,(\mathcal{F}_{n}),\mathbb{P})$ which is uniformly
bounded in $L^{1}$ (i.e. $\sup_{n}\mathbb{E}[|M_{n}|]<\infty$),
and define 
\[
U_{\alpha,\beta}(\omega):=\mbox{the number of upcrossings of }(\alpha,\beta)\mbox{ for }M_{n}(\omega).
\]
Then, one version of the Doob upcrossing inequality says that 
\[
\mathbb{E}[U_{\alpha,\beta}]\leq\frac{\sup_{n}\mathbb{E}[|M_{n}|]+\alpha}{\beta-\alpha}.
\]
Qualitatively, this inequality implies directly that $\mathbb{P}(U_{\alpha,\beta}(\omega)=\infty)=0$
for every $\alpha<\beta$ (which in turn implies that $(M_{n})$ converges
almost surely), but quantitatively it gives us information about the
distribution of the number of upcrossings in terms of $\alpha$ and
$\beta$. By Markov's inequality, we know that $k\cdot\mathbb{P}(\{\omega\mid U_{\alpha,\beta}\geq k\})\leq\mathbb{E}[U_{\alpha,\beta}]$,
which tells us that 
\[
\mathbb{P}(\{\omega\mid U_{\alpha,\beta}\geq k\})\leq\frac{\sup_{n}\mathbb{E}[|M_{n}|]+\alpha}{k(\beta-\alpha)}.
\]
In fact, when combined with the Doob maximal inequality (which says
that off a set of small measure, we can uniformly bound $M_{n}(\omega)$),
this previous bound can be used to deduce a bound on the measure of
the set of points $\omega$ for which $M_{n}(\omega)$ has at least
$k$ $\varepsilon$-fluctuations. Moreover, it can be shown that there
is no uniform rate of convergence in the Martingale convergence theorem,
given only the same initial data as is required by the Doob upcrossing
inequality.\footnote{One easy way to see this is to take a sequence of conditional expectations
$\mathbb{E}(X\vert\mathcal{F}_{n})$, and replace $(\mathcal{F}_{n})$
with a ``slowed down'' filtration like say$(\mathcal{F}_{\lfloor\log n\rfloor})$.
Then the martingale $\mathbb{E}(X\vert\mathcal{F}_{\lfloor\log n\rfloor})$
is still adapted to $(\mathcal{F}_{n})$, since it's always the case
that $\mathcal{F}_{\lfloor\log n\rfloor}\subset\mathcal{F}_{n}$,
but the rate of convergence is exponentially slower than that of $\mathbb{E}(X\vert\mathcal{F}_{n})$. } So this is another example of a theorem which carries uniform convergence
information weaker than a uniform rate of convergence. 

Ending our digression into probability theory, ergodic theoretic statements
of this kind \textemdash{} namely, inequalities which bound the measure
of the set of points in $X$ for which $A_{n}f(x)$ has at least $k$
upcrossings, and thereby bound the measure of the set of points in
$X$ for which $A_{n}f(x)$ has at least $k$ $\varepsilon$-fluctuations,
by way of the maximal ergodic theorem \textemdash{} date back to Bishop's
work on constructive analysis. Given $T\curvearrowright(X,\mu)$ and
$f\in L^{1}(X)$, and letting $E_{\alpha,\beta}(x)$ denote the number
of upcrossings of the interval $(\alpha,\beta)$ of the sequence $A_{n}f(x)$,
Bishop showed \cite{bishop1968constructive} that 
\[
\mu(\{x\mid E_{\alpha,\beta}(x)\geq k\})\leq\frac{\Vert f\Vert_{1}}{k(\beta-\alpha)}.
\]
(The unmistakeable similarity to the Doob upcrossing inequality is
not an accident; Bishop's proof proceeds by proving an abstract upcrossing
inequality which jointly generalizes both the ergodic upcrossing inequality
above, and the Doob upcrossing inequality.) More recently, a similar
upcrossing inequality for $\mathbb{Z}^{d}$ actions where the summation
in the average $A_{n}f$ is taken over symmetric $d$-dimensional
boxes of radius $n$was proved by Kalikow and Weiss \cite{kalikow1999fluctuations}. 

Finally, it is worth mentioning an example of an upcrossing inequality
for a convergence theorem in ergodic theory other than the mean and
pointwise ergodic theorems: recently, Hochman gave an upcrossing inequality
for the Shannon-McMillan-Breiman theorem for $T\curvearrowright(X,\mu)$
\cite{hochman2009upcrossing}.

\chapter{Amenable Groups}

A comprehensive introduction to amenable groups would dwarf the rest
of this document. At the same time, the main result of the following
chapter uses essentially none of the theory of amenability except
for the definition of a Følner sequence. 

What, then, is the purpose of this chapter, if it is neither a self-contained
exposition of the theory of amenable groups, nor a collection of prerequisite
material? Rather, this chapter primarily serves to contextualize the
results of the following chapter. In section 1, we introduce the notion
of a Følner sequence, discuss how it relates to the classical definition
of amenability in terms of finitely additive measures, and give some
illustration of the variety of the class of discrete amenable groups.
In section 2, we discuss briefly how the work of section 1 can be
adapted to the setting of \emph{locally compact} amenable groups.
In section 3, we address the extent to which the ergodic theory of
amenable groups can be viewed as a natural extension of classical
ergodic theory. 

A reader who is already intimately acquainted with geometric group
theory could safely skip the entirety of this chapter, with the notable
exception of Theorem \ref{thm:Amenable-rate} in the final section,
where it is shown that the mean ergodic theorem for amenable groups
has no uniform rate of convergence. This result, though a folk theorem,
is not especially well known, and serves as important motivation for
the thesis as a whole. 

Theorem \ref{thm:Amenable-rate} aside, much of the material in this
chapter is quite standard, and can be found in the books by de la
Harpe \cite{de2000topics}, Dru\c{t}u and Kapovich \cite{drutu2018geometric},
and Einsiedler and Ward \cite{einsiedler2010ergodic}, as well as
the monograph of Anantharaman et al. \cite{anantharaman:hal-00464094}
and the online notes of Juschenko \cite{juschenko2015amenability}
and Tao \cite{tao2009amenable}.

\section{Around Amenability}

The following will serve as our definition of amenability. 
\begin{defn}
A discrete group $G$ has the \emph{Følner property} if, for every
finite set $K$, and every $\varepsilon$, there exists a finite set
$F$ such that for all $k\in K$, 
\[
\frac{|F\Delta kF|}{|F|}<\varepsilon.
\]
Any group with the Følner property is said to be \emph{amenable}. 
\end{defn}

\begin{prop}
If $G$ is countable the Følner property is equivalent to the existence
of a Følner \emph{sequence} $(F_{n})$ for which $|F_{n}\Delta gF_{n}|/|F_{n}|\rightarrow0$
for every $g\in G$. 
\end{prop}

\begin{proof}
$(\Rightarrow$) Let $(g_{n})$ be any enumeration of $G$, and for
every $n$, let $F_{n}$ witness the Følner property for the finite
set $\{g_{1},\ldots,g_{n}\}$ and $\varepsilon=1/n$.

($\Leftarrow$) Given a finite set $F\subset G$ and $\varepsilon>0$,
simply choose $N$ large enough so that for all $n\geq N$ and $g\in F$,
$|F_{n}\Delta gF_{n}|/|F_{n}|<\varepsilon$. 
\end{proof}
Følner sequences will turn out to be the most convenient characterisation
of amenability for our purposes. However, the Følner property is \emph{far
}from the only significant characterization of amenability. When von
Neumann, in his paper \emph{Zur allgemeinen Theorie des Masses} \cite{Neumann1929},
introduced the notion of amenability, he defined a group to be amenable
iff it supports a translation invariant finitely additive probability
measure. While we ultimately make no use of the von Neumann characterization
of amenability in the following chapter, it is worth taking a moment
to illustrate why it is equivalent to the Følner characterization.
\begin{thm}
Let $G$ be a countable discrete group. TFAE:

\begin{enumerate}
\item $G$ has a Følner sequence $(F_{n})$. 
\item $G$ admits a left-invariant finitely additive probability measure.
\item $G$ admits a left-invariant finitely additive mean. 
\end{enumerate}
\end{thm}

\begin{proof}
($1\implies2$) 
{} Given any $B\subseteq G$, consider the limiting behaviour of $|B\cap F_{n}|/|F_{n}|$.
We know that termwise, this ratio is at most 1. Thus we can use an
ultrafilter to fix a limit $\lim_{\mathcal{U}}|B\cap F_{n}|/|F_{n}|$.
Explicitly, let $k$ be an integer. Then partition the unit interval
by 
\[
[0,1/k)\cup[1/k,2/k)\cup\ldots\cup[k-2/k,k-1/k)\cup[k-1/k,1].
\]
Then we can partition $\mathbb{N}$ into subsets $A_{i,k}:=\{m\in\mathbb{N}\mid|B\cap F_{m}|/|F_{m}|\in[i/k,i+1/k)\}$
for $i=0,\ldots,k-2$, and $A_{k-1}=\{m\in\mathbb{N}\mid|B\cap F_{m}|/|F_{m}|\in[k-1/k,1]\}$.
Fix an ultrafilter $\mathcal{U}$ on $\mathbb{N}$. For each $k$,
precisely one of the $A_{i,k}$'s can be an element of $\mathcal{U}$. 

Now we restrict our attention to all $k$'s of the form $2^{j}$.
Then we use the ultrafilter $\mathcal{U}$ as an oracle to answer
denumerably many choices among dyadic intervals, defining $\lim_{\mathcal{U}}|B\cap F_{n}|/|F_{n}|$
to be the unique element of the unit interval which is contained in
the $A_{i,2^{j}}$ belonging to $\mathcal{U}$, for each $j$. 

Setting $\mu(B)=\lim_{\mathcal{U}}|B\cap F_{n}|/|F_{n}|$, we claim
that this is a finitely additive probability measure. To see this,
simply observe that termwise, $|\emptyset\cap F_{n}|/|F_{n}|=0$,
and likewise $|G\cap F_{n}|/|F_{n}|$ is always $1$. If $C$ and
$D$ are disjoint, 
\[
|(C\cup D)\cap F_{n}|=|(C\cap F_{n})\cup(D\cap F_{n})|=|C\cap F_{n}|+|D\cap F_{n}|
\]
and hence 
\[
\frac{|(C\cup D)\cap F_{n}|}{|F_{n}|}=\frac{|C\cap F_{n}|}{|F_{n}|}+\frac{|D\cap F_{n}|}{|F_{n}|}.
\]
Since all these statements hold for every term, we have that $\mu(\emptyset)=0$,
$\mu(G)=1$, and $\mu(C\cup D)=\lim_{\mathcal{U}}(|C\cap F_{n}|/|F_{n}|+|D\cap F_{n}|/|F_{n}|)=\mu(C)+\mu(D)$.
It remains to show that $\mu$ is left-invariant. But this is a consequence
of the fact that $(F_{n})$ is a Følner sequence. To see this, compute
that 
\[
\frac{|B\cap F_{n}|}{|F_{n}|}-\frac{|gB\cap F_{n}|}{|F_{n}|}=\frac{|B\cap F_{n}|-|B\cap g^{-1}F_{n}|}{|F_{n}|}
\]
\[
\leq\frac{|(B\cap F_{n})\backslash(B\cap g^{-1}F_{n})|}{|F_{n}|}=\frac{|B\cap(F_{n}\backslash g^{-1}F_{n})|}{|F_{n}|}\leq\frac{|F_{n}\backslash g^{-1}F_{n}|}{|F_{n}|}
\]
\[
\leq\frac{|F_{n}\Delta g^{-1}F_{n}|}{|F_{n}|}\longrightarrow0
\]
In other words, for every $\varepsilon$, it holds for all but finitely
many $n\in\mathbb{N}$ that $|\frac{|B\cap F_{n}|}{|F_{n}|}-\frac{|gB\cap F_{n}|}{|F_{n}|}|<\varepsilon$.
Hence $\lim_{\mathcal{U}}\frac{|B\cap F_{n}|}{|F_{n}|}-\lim_{\mathcal{U}}\frac{|gB\cap F_{n}|}{|F_{n}|}=0$
and $\mu(B)=\mu(gB)$. 

($2\implies3$) Obvious, since a ``finitely additive mean'' is just
another name for an integral which is defined with respect to a finitely
additive probability measure. 

($3\implies1$) See Tao's notes \cite{tao2009amenable}. 
\end{proof}
\begin{rem*}
The previous result shows that amenability may either be viewed as
a combinatorial or measure theoretic/functional analytic phenomenon.
However, the preceding proof does not allow us to \emph{explicitly
construct} a finitely additive probability measure from a Følner sequence.
This is not an accident: there are models of ZF where $\mathbb{Z}$
does not support a translation-invariant finitely additive probability
measure, but showing that $\mathbb{Z}$ supports a Følner sequence
(which we shall do momentarily) requires only basic arithmetic. 
\end{rem*}
Lots of countable discrete groups are amenable. Let's start at the
beginning:
\begin{prop}
Finite groups are amenable.
\end{prop}

\begin{proof}
Take $F=G$, and observe that $F\Delta kF=0$. 
\end{proof}
\begin{prop}
\label{prop:Z-amenable}$(\mathbb{Z},+)$ is amenable.
\end{prop}

\begin{proof}
If $(m_{i})$ and $(n_{i})$ are sequences of integers such that $m_{i}\leq n_{i}$
for every $i\in\mathbb{N}$, and $n_{i}-m_{i}\rightarrow\infty$,
then $[m_{i},n_{i}]$ is a Følner sequence. To see this, let $k\in\mathbb{Z}$
and compute that 
\[
\frac{|[m_{i},n_{i}]\Delta k[m_{i},n_{i}]|}{|[m_{i},n_{i}]|}=\frac{2|k|}{n_{i}-m_{i}+1}\rightarrow0.
\]
\end{proof}
\begin{prop}
The product of two discrete amenable groups is again amenable. 
\end{prop}

\begin{proof}
Let $G_{1}$ and $G_{2}$ be countable discrete groups with Følner
sequences $(F_{1,n})$ and $(F_{2,n})$. Considering the sequence
$(F_{1,n}\times F_{2,n})$ on $G_{1}\times G_{2}$, we observe that
\[
F_{1,n}\times F_{2,n}\Delta(g_{1},g_{2})F_{1,n}\times F_{2,n}=(F_{1,n}\times F_{2,n})\Delta(g_{1}F_{1,n}\times g_{2}F_{2,n})=(F_{1,n}\Delta g_{1}F_{n})\times(F_{2,n}\Delta g_{2}F_{n})
\]
so that 
\[
\frac{|F_{1,n}\times F_{2,n}\Delta(g_{1},g_{2})F_{1,n}\times F_{2,n}|}{|F_{1,n}\times F_{2,n}|}=\frac{|F_{1,n}\Delta g_{1}F_{1,n}|}{|F_{1,n}|}\frac{|F_{2,n}\Delta g_{2}F_{2,n}|}{|F_{2,n}|}
\]
Thus, if we pick $n$ large enough that both $\frac{|F_{1,n}\Delta g_{1}F_{1,n}|}{|F_{1,n}|}$
and $\frac{|F_{2,n}\Delta g_{2}F_{2,n}|}{|F_{2,n}|}$ are less than
$\sqrt{\varepsilon}$, we see that 
\[
\frac{|F_{1,n}\times F_{2,n}\Delta(g_{1},g_{2})F_{1,n}\times F_{2,n}|}{|F_{1,n}\times F_{2,n}|}<\varepsilon.
\]
Hence $(F_{1,n}\times F_{2,n})$ is a Følner sequence for $G_{1}\times G_{2}$. 
\end{proof}
\begin{prop}
Amenability is invariant under isomorphism. 
\end{prop}

\begin{proof}
Let $\varphi:G_{1}\rightarrow G_{2}$ witness the isomorphism of $G_{1}$
and $G_{2}$. It suffices to show that $\varphi(F_{n})$ is a Følner
sequence. Since $\varphi$ is bijective we know that $|\varphi F_{n}|=|F_{n}|$,
and every element of $G_{2}$ can be written as $\varphi(g)$ for
$g\in G_{1}$. Moreover, $\varphi(g)\varphi(F_{n})=\varphi(gF_{n})$.
Likewise isomorphism commutes with set operations: $\varphi(A\cap B)=\varphi(A)\cap\varphi(B)$,
$\varphi(A\cup B)=\varphi(A)\cup\varphi(B)$, and likewise $\varphi(A^{c})=(\varphi(A))^{c}$,
hence also $\varphi(A\backslash B)=\varphi(A)\backslash\varphi(B)$
and, importantly for us, $\varphi(A\Delta B)=\varphi(A)\Delta\varphi(B)$.
Hence 
\[
\frac{|\varphi F_{n}\Delta\varphi g\varphi F_{n}|}{|\varphi F_{n}|}=\frac{|F_{n}\Delta gF_{n}|}{|F_{n}|}.
\]
\end{proof}
\begin{thm}
Every finitely generated abelian group is amenable.
\end{thm}

\begin{proof}
Using each of the preceding propositions, we use the structure theorem
for finitely generated abelian groups to write $G\cong\mathbb{Z}^{n}\times\prod_{j=1}^{n}\mathbb{Z}/\mathbb{Z}_{q_{j}}$
for natural numbers $q_{j}$. 
\end{proof}
In fact, the previous result extends to \emph{all }countable abelian
groups, by way of the following general fact:
\begin{prop}
Let $(G_{n})$ be a sequence of countable amenable groups such that
$G_{i}\subseteq G_{i+1}$. Then $\bigcup G_{n}$ is also a countable
amenable group.
\end{prop}

\begin{proof}
Obviously $\bigcup G_{n}$ is also a group \textemdash{} given any
two elements $g$ and $h$, there is some index $j$ such that $g,h\in G_{j}$,
hence $gh\in G_{j}\subset G$. The proof of amenability uses the same
strategy. Given any finite subset $K$ of $G$, there is some index
$j$ such that $K\subset G_{j}$. From the amenability of $G_{j}$,
for every $\varepsilon$ there exists a finite $F\subset G_{j}$ such
that $|F\Delta kF|<\varepsilon|F|$, and of course $F$ is also a
subset of $G$.
\end{proof}
\begin{cor}
Since every countable abelian group can be written as a countable
chain of finitely generated abelian groups, we conclude that every
countable abelian group is amenable. 
\end{cor}

As a remark, this does not prove that every countable group is amenable.
One might be tempted to write a countable group as an increasing chain
of finite subsets, but the proof requires that they be finite sub\emph{groups}
\textemdash{} for this to work you'd need to assume that every element
has finite order, which is not true in general!

Before we proceed, we take the opportunity to record several handy
facts about Følner sequences.
\begin{prop}
Let $(F_{n})$ be a Følner sequence on $G$. (1) It is not necessary
that $\bigcup_{n}F_{n}=G$. (2) It is not necessary that for all $n\in\mathbb{N}$,
$F_{n}\subset F_{n+1}$. However, (3) it is always the case that $|F_{n}|\rightarrow\infty$
when $G$ is countably infinite. 
\end{prop}

\begin{proof}
We already proved in Proposition \ref{prop:Z-amenable} that $[m_{i},n_{i}]$
is a Følner sequence provided that $n_{i}-m_{i}\rightarrow\infty$.
This implies that neither (1) nor (2) is necessary.

For (3), suppose $|F_{n}|\leq N$ for all $n\in\mathbb{N}$. We first
suppose that $G$ is finitely generated. If we view $F_{n}$ as a
subset of the Cayley graph, it is clear that there is always at least
one outgoing edge from $F_{n}$. (Otherwise, since the Cayley graph
of $G$ is connected, this would mean that $F_{n}=G$, which does
not occur if $G$ is infinite.) Thus we can pick a generator $g$
of $G$ such that there is an element $f\in F_{n}$ such that $gf\notin F_{n}$.
Consequently, 
\[
\frac{|F_{n}\Delta gF_{n}|}{|F_{n}|}\geq\frac{1}{|F_{n}|}\geq\frac{1}{N}.
\]
This implies that for any generator $g$, there are infinitely many
terms in the Følner sequence such that $|F_{n}\Delta gF_{n}|/|F_{n}|\geq1/N$,
and thus $|F_{n}\Delta gF_{n}|/|F_{n}|\not\rightarrow0$. 

In the case where $G$ is infinitely generated we can run a related
argument. In the Cayley graph, for every point $f\in F_{n}$ it must
be the case that all but finitely many edges from $f$ are outgoing,
simply because $|F_{n}|<\infty$. More specifically, since $|F_{n}|\leq N$
it must be the case that at every point, all but $N-1$ edges are
outgoing. Thus, all but $N(N-1)$ generators are associated to an
edge which is outgoing from \emph{every} point in $F_{n}$.

Consequently, for every $F_{n}$, all but $N(N-1)$ many generators
$g\in G$ have the property that $F_{n}\cap gF_{n}=\emptyset$ and
thus $|F_{n}\Delta gF_{n}|=2|F_{n}|$. 

This implies that all but $N(N-1)$ many generators have the property
that $|F_{n}\Delta gF_{n}|=2|F_{n}|$ for infinitely many $n$ \textemdash{}
observe that if there are $N(N-1)$ many generators which have $|F_{n}\Delta gF_{n}|<2|F_{n}|$
for all but finitely many terms, then there is some index $K$ such
that for all $n\geq K$, \emph{all }of these generators have $|F_{n}\Delta gF_{n}|<2|F_{n}|$,
and consequently for each $n\geq K$ these are the \emph{only} generators
with $|F_{n}\Delta gF_{n}|<2|F_{n}|$.

This shows that there is a $g$ (in fact there are infinitely many)
such that for infinitely many terms in the Følner sequence such that
$|F_{n}\Delta gF_{n}|/|F_{n}|=2$, and thus $|F_{n}\Delta gF_{n}|/|F_{n}|\not\rightarrow0$. 
\end{proof}
\begin{rem*}
If a Følner sequence happens to have the property that $\bigcup_{n}F_{n}=G$,
then we call $(F_{n})$ a \emph{Følner exhaustion}. Likewise if it
so happens that $F_{n}\subset F_{n+1}$ for each $n\in\mathbb{N}$,
we call $(F_{n})$ an \emph{increasing Følner sequence}. Both of these
are frequently occuring side conditions in theorems about amenable
groups.
\end{rem*}
We now give the most basic example of a group which is not amenable.
\begin{prop}
The group $F_{2}$, the free group on two generators, is not amenable. 
\end{prop}

\begin{proof}
Let $K$ be the finite set $\{a,b,a^{-1},b^{-1}\}$. Given another
finite set $F$, let $F_{a}$ denote the subset of words in $F$ beginning
with $a$ and likewise for the other generators. We remark that $F\cap F_{a}$
is a superset of $F\cap aF$ from above \textemdash{} every element
of $aF$ clearly begins with $a$ but need not be an element of $F$.
Observe that 
\[
\frac{|F\Delta gF|}{|F|}+\frac{|F\cap gF|}{|F|}=1
\]
so that amenability is equivalent to being able to find, for every
$\varepsilon$, an $F$ such that for each element $g$ of $\{a,b,a^{-1},b^{-1}\}$
simultaneously, $|F\cap gF|/|F|>1-\varepsilon$. However, 
\[
|F|=|F\cap F_{a}|+|F\cap F_{b}|+|F\cap F_{a^{-1}}|+|F\cap F_{b^{-1}}|.
\]
Thus it is jointly impossible for all of $|F\cap F_{g}|/|F|$ to be
greater than $1/4$, therefore it's impossible for all $|F\cap gF|/|F|$
to be simultaneously greater than $1/4$. 
\end{proof}
\begin{rem*}
The same strategy works for the free group on $n$ generators, just
with $1/2n$ instead of $1/4$. 
\end{rem*}
The class of amenable groups is also closed under the following diagrammatic
operations:
\begin{thm}
(i) Subgroups of amenable groups are amenable. 

(ii) Quotient groups of amenable groups are amenable. 

(iii) Group extensions are amenable: if $N\triangleleft G$ and $N$
and $G/N$ are both amenable, then $G$ is also amenable. 
\end{thm}

\begin{proof}
See Tao's notes \cite{tao2009amenable}. 
\end{proof}
\begin{cor}
Every countable solvable group is amenable. 
\end{cor}

\begin{proof}
Recall that a group $G$ is solvable if there is a finite sequence
$(G_{k})_{k=1,\ldots,n}$ of subgroups of $G$, such that $G_{1}=\{e\}$
and $G_{n}=G$, and such that $G_{k-1}$ is normal in $G_{k}$, and
$G_{k}/G_{k-1}$ is abelian. 

Obviously $\{e\}$ is amenable. Now, suppose that $G_{k-1}$ is amenable.
Then, since $G_{k}/G_{k-1}$ is abelian, and therefore amenable, it
follows from the the third part of the previous theorem that $G_{k}$
is also amenable. Therefore it follows that $G$ is amenable by induction
on $k$.
\end{proof}
\begin{rem*}
Every nilpotent group is solvable, so every nilpotent group is also
amenable. 
\end{rem*}
Before proceeding, we recall the notion of a \emph{word metric} on
a group: given a finitely generated group $G$ with a specified list
of generators, the ``distance'' of an element $g$ to the origin
$e$ is given by the total number of generators in $g$ when $g$
is written as a reduced word (so an element $a^{2}b^{3}$ is distance
5 from the origin, for example). We then say that $d(g,h)$ is given
by the reduced word length of $g^{-1}h$ (a convenient choice which
makes $d$ invariant under left-multiplication). 
\begin{defn}
Consider a (countable) finitely generated group $G$ with a word metric
$d$. We say that $G$ has \emph{subexponential growth }if 
\[
\lim_{n\rightarrow\infty}\frac{\log|\bar{B}(e,n)|}{n}=0
\]
and has exponential growth otherwise. Here, $\bar{B}(e,n)$ denotes
the closed ball of radius $n$ around the identity $e,$ i.e. the
set $\{g\in G\mid d(e,g)\leq n\}$. We sometimes also use the shorthand
$\bar{B}(n)$. 
\end{defn}

\begin{rem*}
The choice of base for the logarithm is irrelevant. Moreover (and
less obviously), the choice of generating set defining the word metric
is \emph{also} irrelevant \textemdash{} this is a consequence of the
fact that word metrics are \emph{quasi-isometric} to each other. See
for instance de la Harpe's book \cite{de2000topics}. 
\end{rem*}
\begin{example}
Let $F_{2}$ be the free group on two generators. For our word metric
we use the generating set $\{a,b,a^{-1},b^{-1}\}$. Then, $\bar{B}(e,1)=5$,
$\bar{B}(e,2)=17$, and more generally $|\bar{B}(e,n+1)|-|\bar{B}(e,n)|=3(|\bar{B}(e,n)|-|\bar{B}(e,n-1)|)$.
By a recursive computation this implies that $|\bar{B}(e,n+1)|-|\bar{B}(e,n)|=3^{n}\cdot4$.
Thus, 
\[
|\bar{B}(e,n+1)|=|\bar{B}(e,0)|+\sum_{k=0}^{n}|\bar{B}(e,k+1)|-|\bar{B}(e,k)|=1+4\sum_{k=0}^{n}3^{k}=1+6(3^{n}-1).
\]
Picking the base of the logarithm as $3$ for convenience, it follows
that 
\[
\frac{\log_{3}|\bar{B}(n)|}{n}=\frac{\log_{3}(2\cdot3^{n+1}-5)}{n}\approx\frac{n+1+\log_{3}2}{n}\longrightarrow1.
\]
Thus, $F_{2}$ has exponential growth as we would expect. A similar
argument works for larger free groups. 
\end{example}

\begin{prop}
Every group of subexponential growth is amenable.
\end{prop}

\begin{proof}
We use the balls under the word metric to satisfy the Følner property. 

Let $K$ be any finite subset of $G$, and fix $\varepsilon$. We
need to find a subset $F$ of $G$ such that $|F\Delta kF|<\varepsilon|F|$
for all $k\in K$. Let $A$ be a finite, symmetric generating set
such that $K\subseteq A$. Thus, $k\bar{B}(n)\subseteq\bar{B}(n+1)$.
However, $|k\bar{B}(n)|=|\bar{B}(n)|$, so we know that on the one
hand $k\bar{B}(n)\backslash\bar{B}(n)$ is a subset of $\bar{B}(n+1)\backslash\bar{B}(n)$,
and on the other hand, we observe that since in general $g(C\backslash D)=gC\backslash gD$,

it follows that $k^{-1}(\bar{B}(n)\backslash k\bar{B}(n))=k^{-1}\bar{B}(n)\backslash\bar{B}(n)$,
which is \emph{also} a subset of $\bar{B}(n+1)\backslash\bar{B}(n)$;
hence, 
\[
\frac{|\bar{B}(n)\Delta k\bar{B}(n)|}{|\bar{B}(n)|}\leq\frac{2(|\bar{B}(n+1)|-|\bar{B}(n)|)}{|\bar{B}(n)|}.
\]
Thus, it suffices to show that for a group of subexponential growth,
for arbitrary $\varepsilon$, $|\bar{B}(N+1)|/|\bar{B}(N)|<1+\varepsilon/2$
for some $N$, so that $\bar{B}(N)$ is the $F$ we're looking for.

To see this, suppose there were some $\varepsilon_{0}$ such that
for all $n$, $|\bar{B}(n+1)|/|\bar{B}(n)|>1+\varepsilon_{0}$. Then,
\[
|\bar{B}(n+1)|>(1+\varepsilon_{0})^{n}.
\]
Thus, $\log|\bar{B}(n+1)|>n\log(1+\varepsilon_{0})$, and 
\[
\lim_{n\rightarrow\infty}\frac{\log|\bar{B}(n)|}{n}>\log(1+\varepsilon_{0})>0
\]
and so $G$ now has exponential growth. 
\end{proof}
Notably, the converse to the previous proposition is false: there
\emph{are} amenable groups with exponential growth, so amenability
does not reduce to the study of the word metric. 
We will give an example of an amenable group of exponential growth
shortly. However, it will be convenient to first introduce another
equivalent characterization of amenability. 
\begin{defn}
(Boundary, $K$-boundary) Let $G$ be a group. Given a subset $F\subset G$,
we say that the \emph{boundary} of $F$ (denoted $\partial F$) is
the set of all points $g\in G$ such that $\bar{B}(g,1)\cap F\neq\emptyset$
and also $\bar{B}(g,1)\cap F^{C}\neq\emptyset$. More generally, given
a finite subset $K$ of $G$, the $K$-boundary of $F$ (denoted $\partial_{K}F$)
is the set of all points $g\in G$ such that $Kg\cap F\neq\emptyset$
and also $Kg\cap F^{C}\neq\emptyset$. (Evidently $\partial F=\partial_{\bar{B}(e,1)}F$.) 
\end{defn}

\begin{rem*}
It is sometimes helpful to note that $|\partial F|$ is at most $2$
times the number of outgoing edges from $F$ in the Cayley graph of
$G$. In fact, some sources define $\partial F$ as the set of outgoing
edges from $F$ in the Cayley graph of $G$, since the combinatorial/geometric
role of the two notions is nearly the same. 

Especially with this latter definition of $\partial F$, the choice
of the term ``boundary'' is intended to emphasize the fact that,
in the discrete geometry of a (Cayley graph of a) finitely generated
group, the set of outgoing edges from a subset $F$ really does play
a similar role to the boundary of a subset of space in a more conventional
setting. For instance, with this metaphor in hand, we can define the
\emph{isoperimetric problem }for groups, where we seek to find, for
a fixed cardinality of $\partial F$, what is the greatest possible
cardinality of $F$. (Recall that we usually think of an isoperimetric
problem as looking to maximize the volume enclosed by an oriented
surface of a given surface area.) For an \emph{isoperimetric inequality}
for groups, see Theorem 5.11 in Pete's book \cite{pete2015probability};
for a treatment of the isoperimetric problem for groups which emphasizes
the analogy with isoperimetric problems in other geometric settings,
see the book by Figalli et al. \cite{figalli2011autour}. 
\end{rem*}
\begin{prop}
\emph{(Boundary characterization of amenability) }Suppose that $G$
is a discrete group. TFAE:

\begin{enumerate}
\item G is amenable.
\item For every finite $K\subset G$ and $\varepsilon>0$, there exists
a finite $F\subset G$ such that $|\partial_{K}F|/|F|<\varepsilon$. 
\end{enumerate}
Suppose moreover that $G$ is countable. Then $(F_{n})$ is a Følner
sequence iff, for all finite $K\subset G$, $|\partial_{K}F_{n}|/|F_{n}|\rightarrow0$. 
\end{prop}

\begin{proof}
See section I.1 of Ornstein and Weiss \cite{ornstein1987entropy}
and/or lemma 2.6 of Pogorzelski and Schwarzenberger \cite{pogorzelski2016banach}. 
\end{proof}
\begin{example}
\label{exa:Baumslag}The Baumslag-Solitar group $BS(1,2)$, namely
the group on two generators characterized by the presentation $\langle a,b\mid bab^{-1}=a^{2}\rangle$,
is a well-known example of a group of exponential growth which is
solvable, and therefore amenable.\footnote{More generally, the family of Baumslag-Solitar group\emph{s} $BS(m,n):=\langle a,b\mid ba^{m}b^{-1}=a^{n}\rangle$
is a well-known family of pathological/counterexample objects. Aside
from being a solvable group of exponential growth, $BS(1,2)$ was
recently shown to be \emph{scale-invariant} \cite{nekrashevych2011scale},
thus disproving a conjecture of Itai Benjamini that scale-invariant
groups always have polynomial growth. }

\begin{figure}
\begin{centering}
\label{BS12}\includegraphics[scale=0.35]{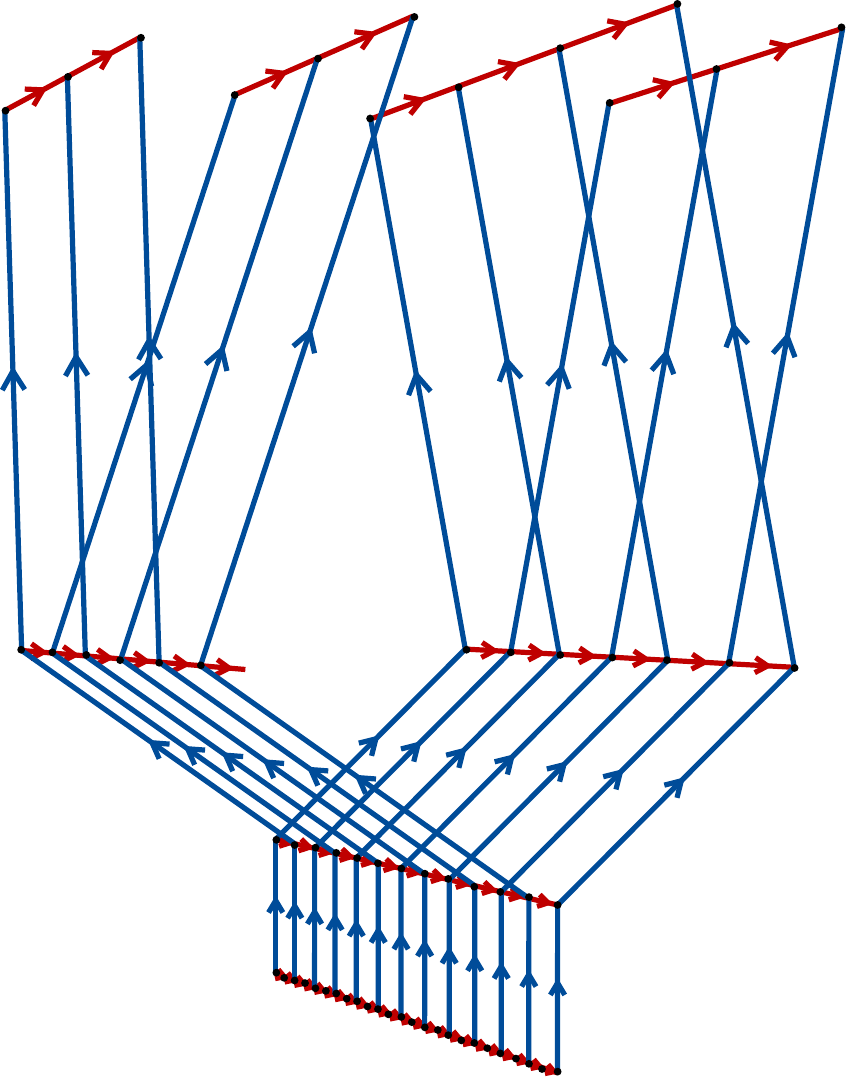}$\qquad$\includegraphics[scale=0.35]{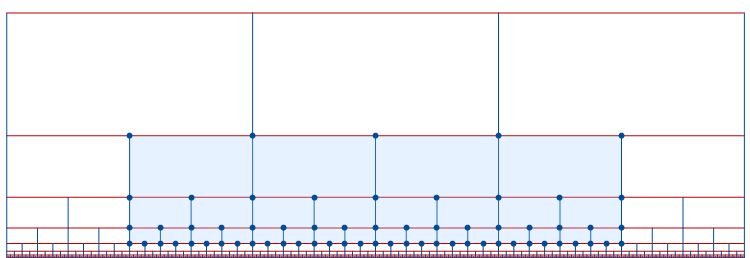}
\par\end{centering}
\caption{At left, part of the Cayley graph of $BS(1,2)$ in its standard 3D
embedding. At right, a single ``sheet'' of the group; the highlighted
region indicates the portion of the ``wide rectangle'' $R_{n,m}$
which lies in the given sheet. In both images, the colour-coding indicates
the ``coordinate system'' of $BS(1,2)$ in terms of the generators:
from a given vertex, moving up corresponds to right multiplication
by $b$, and moving right corresponds to right multiplication by $a$.
(Photo credit Jim Belk \cite{bs12}.) }
\end{figure}

One can also give an explicit description of a Følner sequence for
$BS(1,2)$; a natural way to do so in this case is to use the boundary
characterization of amenability. (Much of the discussion that follows
adheres closely to Belk's exposition \cite{bs12}, which is also our
source for the associated figure.) First, note that (the Cayley graph
of) $BS(1,2)$ has a canonical embedding in 3 space, as depicted in
Figure 2.1, which also describes the ``coordinate system'' for $BS(1,2)$
in terms of the generators $a$ and $b$. ``Rectangles'' in $BS(1,2)$
shall be defined as follows: a point $g$ belongs to the rectangle
$R_{m,n}$ if, starting from the origin, we can reach $g$ by first
traveling down $n$ edges (corresponding to right multiplication by
$b^{-n}$), then traveling left or right along at most $m$ edges
(corresponding to right multiplication by $a^{k}$ with $k\in[-m,m]$),
and then traveling up at most $2n$ edges (corresponding to right
multiplication by $b^{j}$ with $j\in[0,2n]$). Thus, a more algebraic
way to write $R_{m,n}$ is as the set 
\[
R_{m,n}:=\{b^{-n}a^{k}b^{j}\mid k\in[-m,m],j\in[0,2n]\}.
\]
It is clear that $|R_{m,n}|=(2m+1)(2n+1)$. Likewise, $R_{m,n}$ has
$(2m+1)$ boundary edges on the top and bottom ``sides''. However,
The left and right sides of $R_{m,n}$ are actually shaped like a
binary tree of height $2n+1$ (provided that $m$ is divisible by
$2^{2n}$, otherwise the sides will not ``fully branch''; in any
case this is a satisfactory upper bound), and thus the number of edges
on each side is $\sum_{j=0}^{2n}2^{j}=2^{2n+1}.$ So compute (using
the boundary edge estimate for $|\partial R_{m,n}|$) that 
\[
\frac{|\partial R_{m,n}|}{|R_{m,n}|}\leq2\frac{2(2m+1)+2(2^{2n+1})}{(2m+1)(2n+1)}.
\]
Evidently the relative boundary size will be small provided that $m$
is exponentially bigger than $n$. For instance, if we take the rectangle
$R_{2^{2k},k}$, we have that 
\[
\frac{|\partial R_{2^{2k},k}|}{|R_{2^{2k},k}|}\leq2\frac{2(2\cdot2^{2k}+1)+2(2^{2k+1})}{(2\cdot2^{2k}+1)(2k+1)}=2\frac{2}{2k+1}+2\frac{2(2\cdot2^{2k})}{(2\cdot2^{2k}+1)(2k+1)}<\frac{4}{k}.
\]
Hence $|\partial R_{2^{2k},k}|/|R_{2^{2k},k}|\longrightarrow0$, and
$(R_{2^{2k},k})$ is a Følner sequence for $BS(1,2)$. 

We can also show that $(R_{2^{2k},k})$ is a Følner sequence, in the
conventional sense. Indeed, let $g\in BS(1,2)$. As a reduced word,
$g$ corresponds to a product of $a$'s, $b$'s, $a^{-1}$s, and $b^{-1}$s;
on the Cayley graph, $gR_{2^{2k},k}$ corresponds to applying a series
of shifts right, up, left, and down respectively. For instance, $|R_{2^{2k},k}\Delta aR_{2^{2k},k}|$
is just 2 times the number of boundary edges coming out of the right
side of $R_{2^{2k},k}$ (which we already saw was $2^{2k+1}$); likewise
$|R_{2^{2k},k}\Delta bR_{2^{2k},k}|=2(2^{2k}+1)$ based on our previous
computation of the number of edges at the top side of $R_{2^{2k},k}$. 

If we apply a series of shifts to $R_{2^{2k},k}$, we can estimate
$|R_{2^{2k},k}\Delta gR_{2^{2k},k}|$ by the triangle inequality for
the symmetric difference $\Delta$: $|F\Delta h_{2}h_{1}F|\leq|F\Delta h_{1}F|+|h_{1}F\Delta h_{2}F|$.
Since a shifted copy of $R_{2^{2k},k}$ has the same combinatorial
properties as $R_{2^{2k},k}$, this implies that 
\[
|R_{2^{2k},k}\Delta gR_{2^{2k},k}|\leq2N(2^{2k+1})+2M(2^{2k}+1)
\]
where $N$ is the number of shifts in the horizontal direction, and
$M$ is the number of shifts in the vertical direction, in the reduced
word of $g$. 

To estimate $N$ and $M$, observe that the rectangles $R_{m,n}$
exhaust the group as $n,m\rightarrow\infty$. In other words there
is some rectangle $R_{m,n}$ which contains $g$. Thus $g$ can be
written in the form $b^{-n}a^{k}b^{j};\,k\in[-m,m],j\in[0,2n]$. Turning
this around slightly, if $g\in R_{m,n}$ then $N\leq2n$ and $M\leq2m$.
It follows that for all $g\in R_{m,n}$, 
\[
\frac{|R_{2^{2k},k}\Delta gR_{2^{2k},k}|}{|R_{2^{2k},k}|}\leq\frac{2n(2^{2k+1})+2m(2^{2k}+1)}{(2\cdot2^{2k}+1)(2k+1)}=\frac{2n+2m}{2k+1}.
\]
Evidently, as $k\rightarrow\infty$, $|R_{2^{2k},k}\Delta gR_{2^{2k},k}|/|R_{2^{2k},k}|\rightarrow0$. 

(We will make use of this estimate again in the next chapter.)
\end{example}

\section{A Word on Locally Compact Amenable Groups}

Thus far we have focused on amenable groups which are countable and
discrete. It is also possible to adapt the notion of amenability to
locally compact topological groups.
\begin{defn}
A locally compact topological group $(G,\tau)$ (with Haar measure
$m_{G}$) is said to be amenable if, for every \emph{compact} set
$K$, and every $\varepsilon$, there is a \emph{compact} set $F$
and a set $K_{0}\subset K$ with $m_{G}(K\backslash K_{0})<\varepsilon$,
such that for all $k\in K_{0}$, 
\[
\frac{m_{G}(F\Delta kF)}{m_{G}(F)}<\varepsilon.
\]
\end{defn}

Broadly, ``finite'' for discrete amenable groups is replaced with
``compact'', the counting measure is replaced with the Haar measure,
and countability is replaced with the assumption that the topology
is $\sigma$-compact, or equivalently is second countable. (The assumption
that $G$ is finitely generated, required in proofs which exploit
the word metric, is replaced with the assumption that $(G,\tau)$
is \emph{compactly} generated. Notably, this means that the Haar measure
of $\bar{B}(n)$ is always finite under the word metric.) With these
replacements, many proofs carry over \emph{mutatis mutandis}. For
instance, it is possible to prove that $\mathbb{R}$ is amenable in
the same way that we proved $\mathbb{Z}$ is amenable. The proof that
products of amenable groups are again amenable is identical, except
for the replacement of the counting measure $|\cdot|$ with $m_{G}$.
Topological groups which are\emph{ }compact (rather than finite) are
again trivially amenable. Since the structure theory of locally compact
abelian groups tells us that every locally compact, compactly generated
abelian group decomposes as a product $\mathbb{R}^{d}\times\mathbb{Z}^{\ell}\times K$
where $d,\ell\in\mathbb{N}$ and $K$ is compact, we see that every
finitely generated locally compact compactly generated abelian group
is amenable. And so on. 

This heuristic does not hold in utmost generality; for instance, the
boundary characterization of amenability is only valid for locally
compact groups which are \emph{unimodular}. 

Much of the translation between discrete and locally compact notions
in the theory of amenability (and geometric group theory more generally)
is folk theory, but two helpful references are the monograph by Ornstein
and Weiss \cite{ornstein1987entropy}, and the recent book by Cornulier
and de la Harpe \cite{cornulier2016metric}. 

\section{Ergodic Theory and Amenable Groups}

It is now apparent that numerous results in classical ergodic theory
\textemdash{} namely, wherein one studies the action of a single measure-preserving
transformation on a probability space \textemdash{} have natural analogues
if the action of a single transformation is replaced with the action
of an amenable group. 

To give a small amount of motivation, first observe that if we have
a measure-preserving action of $\mathbb{Z}$ on a space $(X,\mu)$,
this is precisely the same as having the action of a single invertible
measure-preserving transformation $T$, where $T^{n}x=n\cdot x$.
Likewise, a measure-preserving action of $\mathbb{Z}^{d}$ on $(X,\mu)$
can also be described as the action of $d$ distinct invertible transformations
on $(X,\mu)$, provided that all of these transformations commute
with each other. 

Likewise, a common proof technique in ergodic theory goes as follows:
we want to approximate $\frac{1}{N}\sum_{i=0}^{N-1}f\circ T^{i}$
with $\frac{1}{N}\sum_{i=0}^{N-1}(f\circ T^{k})\circ T^{i}$. To do
this, we observe that this latter sum is equal to $\frac{1}{N}\sum_{i=k}^{N-1+k}f\circ T^{i}$,
and note that if $N\gg k$ then the difference between the two sums
becomes very small, since all but $2k$-many terms cancel. Ultimately,
this exploits the fact that $[0,N-1]$ is a Følner sequence in $\mathbb{Z}$:
for every $k\in\mathbb{Z}$ and $\varepsilon>0$, we can pick an $N$
such that $|[0,N-1]\Delta k\cdot[0,N-1]|/|[0,N-1]<\varepsilon$. 

Indeed, if we have a countable discrete (or locally compact second
countable) amenable group $G$ acting on a space $(X,\mu)$, we can
define the ``amenable ergodic average'' 
\[
\frac{1}{|F_{n}|}\sum_{\gamma\in F_{n}}f\circ\gamma^{-1}\quad(G\mbox{ countable});\quad\frac{1}{m_{G}(F_{n})}\int_{F_{n}}f\circ\gamma^{-1}dm_{G}\quad(G\mbox{ second countable})
\]
where $(F_{n})$ is any Følner sequence for $G$. That this is the
right generalization of classical ergodic averages should be at least
suggested by the proof of the following theorem. 
\begin{thm}
\emph{\label{thm:(discrete amenable MET)}(Mean ergodic theorem for
countable discrete amenable groups)} Let $G$ be a countable discrete
amenable group acting by unitary transformations on a Hilbert space
$H$ via some representation $\pi$, let $(F_{n})$ be a Følner sequence
for $G$, and let $f\in H$. Let $P_{G}$ denote the orthogonal projection
to the subspace of $H$ which is invariant under the action of $\gamma$
for every $\gamma\in G$. Then, $\frac{1}{|F_{n}|}\sum_{\gamma\in F_{n}}\pi(\gamma^{-1})f$
converges to $P_{G}f$ in the norm of $H$. 

In particular, if $f\in L^{2}(X,\mu)$, it follows (from the Koopman
formalism) that $\frac{1}{|F_{n}|}\sum_{\gamma\in F_{n}}f\circ\gamma^{-1}$
converges to $P_{G}f$ in the $L^{2}$ norm. 
\end{thm}

\begin{rem*}
In fact, the same result holds if the acting group is $\sigma$-compact
locally compact with a Haar measure rather than countable and discrete,
and this is also how the theorem is stated in Theorem 8.13 of Einsiedler
and Ward's \emph{Ergodic Theory: with a view towards Number Theory}
\cite{einsiedler2010ergodic}. Of course the version as stated above
is a special case. 

As in the common textbook proof of the von Neumann mean ergodic theorem,
it is easier to work in the more abstract setting of unitary representations
and Hilbert spaces than to work directly with an action on a measure
space. 
\end{rem*}
\begin{proof}
Suppose $f$ is $G$-invariant. Then clearly for each $n$, 
\[
\frac{1}{|F_{n}|}\sum_{\gamma\in F_{n}}\pi(\gamma^{-1})f=\frac{1}{|F_{n}|}\left(|F_{n}|\cdot f\right)=f
\]
so in this case the result holds trivially. Moreover, since the sum
of two $G$-invariant elements of $H$ is again $G$-invariant, and
likewise multiplication by a scalar respects $G$-invariance (and
the function $0$ is trivially $g$-invariant), the $G$-invariant
elements form a subspace in the Hilbert space $H$. Denote this space
by $\mathcal{I}$. 

Moreover, we readily see that $\mathcal{I}$ is closed. Let $(f_{n})$
be a sequence in $\mathcal{I}$ converging to $f$. Then, since in
general $||\pi(\gamma)f||=||f||$, 
\[
||f_{n}-\pi(\gamma)f||=||\pi(\gamma)f_{n}-\pi(\gamma)f||=||\pi(\gamma)(f_{n}-f)||=||f_{n}-f||\rightarrow0.
\]
Thus $f_{n}$ converges simultaneously to $f$ and $\pi(\gamma)f$,
and so they are equal. Hence $f$ is also $G$-invariant; so $\mathcal{I}$
is closed. 

Likewise, consider the space $\mathcal{N}$ defined by taking the
closure
{} of the subspace spanned by all points of the form $\{f-\pi(\gamma)f\}$
for all $f\in H$, $\gamma\in G$. (These are sometimes called \emph{coboundary
terms}.) We claim that this is the orthogonal complement of the space
of $G$-invariant elements. Evidently if $g$ is $G$-invariant then,
$g=\pi(\gamma)g$, so since the action of $G$ is unitary, 
\[
\langle g,f-\pi(\gamma)f\rangle=\langle g,f\rangle-\langle g,\pi(\gamma)f\rangle=\langle g,f\rangle-\langle\pi(\gamma)g,\pi(\gamma)f\rangle=0
\]
so by a density argument, if $h\in\mathcal{N}$ then $\langle g,h\rangle\leq||g||\varepsilon$
for every $\varepsilon$, hence $g\in\mathcal{N}^{\bot}$. Thus $\mathcal{N}^{\bot}$
contains the invariant subspace $\mathcal{I}$. Conversely, suppose
that for all $f\in H$, $\langle g,f-\pi(\gamma)f\rangle=0$. Then
$\langle g,f\rangle=\langle g,\pi(\gamma)f\rangle$. Since the action
of $G$ is unitary, it also holds that $\langle\pi(\gamma^{-1})g,f\rangle=\langle g,\pi(\gamma)f\rangle$.

But in a Hilbert space, 
\[
[\forall f.\langle g,f\rangle=\langle\pi(\gamma^{-1})g,f\rangle]\implies g=\pi(\gamma^{-1})g.
\]
Equivalently, $g=\pi(\gamma)g$. Thus, if we now quantify over all
$\gamma\in G$, we see that 
\[
[\forall\gamma.\forall f.\langle g,f-\pi(\gamma)f\rangle=0]\implies\forall\gamma.g=\pi(\gamma)g.
\]
Hence, if $g$ is orthogonal to the spanning set $\{f-\pi(\gamma)f\}$
generating $\mathcal{N}$ (and therefore, $g\in\mathcal{N}^{\bot}$)
then $g\in\mathcal{I}$. Thus $\mathcal{N}^{\bot}$ is the $G$-invariant
subspace. In particular we have $H=\mathcal{I}\oplus\mathcal{N}$.

Pick any $f$ in this subspace, i.e. any function of the form $\sum_{j=1}^{k}c_{j}(g_{j}-\pi(\gamma_{j}^{-1})g_{j})+g_{\varepsilon}$
where $||g_{\varepsilon}||<\varepsilon$. 
\begin{align*}
||\frac{1}{|F_{n}|}\sum_{\gamma\in F_{n}}\pi(\gamma^{-1})f|| & =||\frac{1}{|F_{n}|}\sum_{\gamma\in F_{n}}\sum_{j=1}^{k}c_{j}(\pi(\gamma^{-1})g_{j}-\pi((\gamma_{j}\gamma)^{-1})g_{j}+\frac{1}{|F_{n}|}\sum_{\gamma\in F_{n}}\pi(\gamma^{-1})g_{\varepsilon})||\\
 & \leq||\frac{1}{|F_{n}|}\sum_{j=1}^{k}c_{j}\left(\sum_{\gamma\in F_{n}}\pi(\gamma^{-1})g_{j}-\sum_{\beta\in\gamma_{k}F_{n}}\pi(\beta^{-1})g_{j}\right)||+\frac{1}{|F_{n}|}\sum_{\gamma\in F_{n}}||\pi(\gamma^{-1})g_{\varepsilon}||\\
 & \leq\frac{1}{|F_{n}|}\sum_{j=1}^{k}c_{j}\left(\sum_{\gamma\in F_{n}\Delta\gamma_{k}F_{n}}||\pi(\gamma^{-1})g_{j}||\right)+||g_{\varepsilon}||\\
 & \leq\frac{1}{|F_{n}|}\sum_{j=1}^{k}c_{j}\left(\sum_{\gamma\in F_{n}\Delta\gamma_{k}F_{n}}||g_{j}||\right)+\varepsilon\\
 & =\sum_{j=1}^{k}c_{j}\frac{|F_{n}\Delta\gamma_{k}F_{n}|}{|F_{n}|}||g_{j}||+\varepsilon
\end{align*}
By amenability, we can pick $N\in\mathbb{N}$ such that for all $n\geq N$,
and every $j=1,\ldots,k$, 
\[
\frac{|F_{n}\Delta\gamma_{k}F_{n}|}{|F_{n}|}<\frac{\varepsilon}{\sum_{j=1}^{k}c_{j}||g_{j}||}
\]
so that $||\frac{1}{|F_{n}|}\sum_{\gamma\in F_{n}}\pi(\gamma^{-1})f||<2\varepsilon$. 

Using orthogonal decomposition, we then take any $f\in H$ and uniquely
write it as the sum $P_{G}f+f_{\bot}$ where $P_{G}f$ is the projection
to the $G$-invariants and $f_{\bot}\in\mathcal{N}$. In turn, for
any $\varepsilon$ we can always decompose $f_{\bot}=\sum_{j=1}^{k}c_{j}(g_{j}-\pi(\gamma_{j}^{-1})g_{j})+g_{\varepsilon}$
with $||g_{\varepsilon}||<\varepsilon$. Then by the previous calculation,
\begin{align*}
||P_{G}f-\frac{1}{|F_{n}|}\sum_{\gamma\in F_{n}}\pi(\gamma^{-1})f|| & \leq||P_{G}f-\frac{1}{|F_{n}|}\sum_{\gamma\in F_{n}}\pi(\gamma^{-1})(P_{G}f)||+||\frac{1}{|F_{n}|}\sum_{\gamma\in F_{n}}\pi(\gamma^{-1})f_{\bot}||\\
 & \leq\sum_{j=1}^{k}c_{j}\frac{|F_{n}\Delta\gamma_{j}F_{n}|}{|F_{n}|}||g_{j}||+\varepsilon\\
\implies\lim_{n\rightarrow\infty}||P_{G}f-\frac{1}{|F_{n}|}\sum_{\gamma\in F_{n}}\pi(\gamma^{-1})f|| & <2\varepsilon.
\end{align*}
Finally, we send $\varepsilon$ to zero. 
\end{proof}
The reader should observe that the preceding proof is nearly word-for-word
identical with the common proof of the von Neumann mean ergodic theorem,
except that the sequence $[0,n)$ of intervals in $\mathbb{Z}$ has
been replaced with a Følner sequence, and the projection onto the
$T$-invariant (equivalently, $\mathbb{Z}$-invariant!) subspace is
now a projection onto the $G$-invariant subspace. 

As in the classical setting, one can show that there is no uniform
rate of convergence in the amenable mean ergodic theorem \textemdash{}
in fact one can show something stronger, namely that for any \emph{fixed}
amenable group there is no uniform rate of convergence. However, this
is actually a case where the machinery of amenable groups allows for
a significantly streamlined argument. 
\begin{thm}
\label{thm:Amenable-rate}Given a locally compact second countable
amenable group $G$, there exists a Hilbert space $H$ and an action
of $G$ on $H$ via unitary representation $\pi$, such that for any
Følner sequence $(F_{n})$ on $G$, there is no uniform rate of convergence
for the family of sequences $\left\{ \frac{1}{m_{G}(F_{n})}\int_{F_{n}}\pi(\gamma^{-1})fdm_{G};f\in H\right\} $.
\end{thm}

\begin{proof}
A convenient choice of $H$ and $\pi$ is $L^{2}(G,m_{G})$ with precomposition
by left-multiplication (i.e. $\pi(g)f(x):=f(gx)$). Let $(\alpha_{n})$
be a decreasing sequence of positive reals encoding a rate of convergence.
Without loss of generality, $\alpha_{n}<1$ for all $n$. Ultimately,
given an $n\in\mathbb{N}$, it suffices to find an $f\in L^{2}(G)$
such that $\Vert A_{n}f-P_{G}f\Vert_{2}>\alpha_{n}$. 

First, given any measurable subset $B$ of $G$, we define the normalized
characteristic function $\bar{\mathbf{1}}_{B}:=\mathbf{1}_{B}/\left(m_{G}(B)\right)^{1/2}$,
so that $\Vert\bar{\mathbf{1}}_{B}\Vert_{2}=1$. 

First, notice that for a fixed $h\in G$, 
\[
\Vert\pi(h^{-1})\bar{\mathbf{1}}_{B}-\bar{\mathbf{1}}_{B}\Vert_{2}^{2}=\int_{G}\left(\frac{\bar{\mathbf{1}}_{B}(h^{-1}g)-\bar{\mathbf{1}}_{B}(g)}{\left(m_{G}(B)\right)^{1/2}}\right)^{2}dm_{G}(g)\leq\frac{m_{G}(h^{-1}B\Delta B)}{m_{G}(B)}.
\]
Now, given $F_{n}$ and $\varepsilon>0$ (with $\alpha_{n}<1-\varepsilon$),
we can use the Følner property to find some compact $B$ and a subset
$F_{n}^{\prime}\subset F_{n}$ such that for all $h\in F_{n}^{\prime}$,
{} 
\[
\frac{m_{G}(h^{-1}B\Delta B)}{m_{G}(B)}<\varepsilon^{2}\cdot m_{G}(F_{n})/9;\quad m_{G}(F_{n}\backslash F_{n}^{\prime})<\varepsilon^{2}\cdot m_{G}(F_{n})/9.
\]
(To be picky, the Følner property actually tells us that $\frac{m_{G}(hB\Delta B)}{m_{G}(B)}<\varepsilon^{2}\cdot m_{G}(F_{n})/9$.
However, left invariance tells us that $m_{G}(hB\Delta B)=m_{G}(B\Delta h^{-1}B)$.) 

Additionally, let $k$ be a ``large enough'' element of $G$ so
that $B$ and $Bk$ are disjoint. (Such a $k$ always exists since
$B$ is compact and $G$ is not.)
{} Write 
\[
f=\frac{1}{2}\left(\bar{\mathbf{1}}_{B}-\bar{\mathbf{1}}_{Bk}\right)
\]
so that $\int fdm_{G}=0$ but $\Vert f\Vert_{2}=1$. Notably, since
$G$ acts ergodically on itself (!), we know that $A_{n}f\stackrel{L^{2}(G)}{\longrightarrow}\int fdm_{G}$.
Therefore, it suffices to show that $\Vert A_{n}f-f\Vert_{2}<\varepsilon$,
since this implies $\Vert A_{n}f-0\Vert_{2}>1-\varepsilon>\alpha_{n}$.

Observe that 
\[
\Vert\pi(h^{-1})f-f\Vert_{2}\leq\frac{1}{2}\Vert\pi(h^{-1})\bar{\mathbf{1}}_{B}-\bar{\mathbf{1}}_{B}\Vert_{2}+\frac{1}{2}\Vert\pi(h^{-1})\bar{\mathbf{1}}_{Bk}-\bar{\mathbf{1}}_{Bk}\Vert_{2}
\]
and that if $\delta(k)$ denotes the Haar modular character for $m_{G}$,
\[
\frac{m_{G}(h^{-1}Bk\Delta Bk)}{m_{G}(Bk)}=\frac{\delta(k)m_{G}(h^{-1}B\Delta B)}{\delta(k)m_{G}(B)}=\frac{m_{G}(h^{-1}B\Delta B)}{m_{G}(B)}
\]
so \emph{both }$\Vert\pi(h^{-1})\bar{\mathbf{1}}_{B}-\bar{\mathbf{1}}_{B}\Vert_{2}$
and $\Vert\pi(h^{-1})\bar{\mathbf{1}}_{Bk}-\bar{\mathbf{1}}_{Bk}\Vert_{2}$
are less than $\varepsilon\cdot\sqrt{m_{G}(F_{n})}/3$, and therefore
\[
\Vert\pi(h^{-1})f-f\Vert_{2}<\varepsilon\cdot\sqrt{m_{G}(F_{n})}/3.
\]

Now, compute that 
\begin{align*}
\Vert A_{n}f-f\Vert_{2}^{2} & =\int_{G}\left(\frac{1}{m_{G}(F_{n})}\int_{F_{n}}f(h^{-1}g)dm_{G}(h)-f(g)\right)^{2}dm_{G}(g)\\
 & =\int_{G}\left(\frac{1}{m_{G}(F_{n})}\int_{F_{n}}\left(f(h^{-1}g)-f(g)\right)dm_{G}(h)\right)^{2}dm_{G}(g)\\
\mbox{(Jensen)} & \leq\int_{G}\frac{1}{\left(m_{G}(F_{n})\right)^{2}}\int_{F_{n}}\left(f(h^{-1}g)-f(g)\right)^{2}dm_{G}(h)dm_{G}(g)\\
\mbox{(Fubini)} & =\frac{1}{\left(m_{G}(F_{n})\right)^{2}}\int_{F_{n}}\int_{G}\left(f(h^{-1}g)-f(g)\right)^{2}dm_{G}(g)dm_{G}(h).
\end{align*}
We split $F_{n}$ into $F_{n}^{\prime}$ and $F_{n}\backslash F_{n}^{\prime}$.
On $F_{n}^{\prime}$, we know that $\Vert\pi(h^{-1})f-f\Vert_{2}^{2}<\varepsilon^{2}\cdot m_{G}(F_{n})/9$,
and on $F_{n}\backslash F_{n}^{\prime}$ we use the crude bound $\Vert\pi(h^{-1})f-f\Vert_{2}^{2}\leq2\Vert f\Vert_{2}^{2}=2$.
Hence 
\begin{align*}
\int_{F_{n}}\int_{G}\left(f(h^{-1}g)-f(g)\right)^{2}dm_{G}(g)dm_{G}(h) & <\int_{F_{n}^{\prime}}\varepsilon^{2}\cdot m_{G}(F_{n})/9dm_{G}(h)+\int_{F_{n}\backslash F_{n}^{\prime}}2dm_{G}(h)\\
 & <\varepsilon^{2}m_{G}(F_{n})^{2}/9+2\varepsilon^{2}m_{G}(F_{n})/9.
\end{align*}
Consequently, $\Vert A_{n}f-f\Vert_{2}^{2}<\varepsilon^{2}/9+2\varepsilon^{2}/(3m_{G}(F_{n}))$.
Since $m_{G}(F_{n})\rightarrow\infty$ for any Følner sequence, without
loss of generality $m_{G}(F_{n})\geq1$, so that $\Vert A_{n}f-f\Vert_{2}^{2}<7\varepsilon^{2}/9$
and thus $\Vert A_{n}f-f\Vert_{2}<\varepsilon$. 

It is worth noting that, \emph{mutatis mutandis}, the same argument
works if we replace the exponent $2$ with any $p\in[1,\infty)$. 
\end{proof}
\begin{rem*}
(for the reader who is familiar with Kazhdan groups and the like)
The previous proof is essentially ``just'' an application of the
fact that the left contravariant action of $G$ on $L^{2}(G)$ admits
\emph{almost-invariant vectors} provided that $G$ is amenable. It
is \emph{not} a coincidence that such a proof does not go through
for Kazhdan groups, which \emph{never} have almost-invariant vectors
in this setting. Indeed the spectral gap characterization of Kazhdan
groups can sometimes be exploited to give a uniform rate of convergence
for a mean ergodic theorem (see for instance Gorodnik and Nevo's survey
article \cite{gorodnik2015quantitative}). 
\end{rem*}
A large enough portion of classical ergodic theory has now been ``amenable-ized''
(including, notably, the entire machinery of Ornstein isomorphism
theory \cite{ornstein1987entropy}) that it is tempting to form the
heuristic that given any theorem involving a measure-preserving $\mathbb{Z}$-action,
there will be \emph{some} analogous theorem where $\mathbb{Z}$ is
replaced with an amenable group. However, it is worth remarking that
many \emph{proofs} in classical ergodic theory do \emph{not} adapt
to the amenable setting as readily as in the preceding proof of the
mean ergodic theorem, nor does the amenable setting always offer us
a ``nicer'' proof as in the preceding proof of the lack of a uniform
rate of convergence for the amenable MET. 

To give a concrete example, one of the standard proofs of the Birkhoff
ergodic theorem (given, for instance, in Einsiedler and Ward's book)
proves the maximal ergodic theorem via a Vitali covering argument
on $\mathbb{Z}$, and then combines the maximal ergodic theorem and
the mean ergodic theorem to deduce pointwise a.s convergence. It so
happens that in the countable discrete setting, the same Vitali covering
argument generalized naturally to an action of any group $G$ which
has \emph{polynomial growth }(a large subclass of amenable groups,
identical by a result of Gromov to the class of all virtually nilpotent
groups), but fails to generalize directly to all amenable groups;
and the proof of the pointwise ergodic theorem for \emph{arbitrary}
second countable amenable groups, due to Lindenstrauss, ultimately
relies on a novel and sophisticated replacement for the Vitali covering
argument. 

In some notable cases, the best known generalization of a result in
classical ergodic theory only covers a very small sub-class of amenable
groups: for instance, the best generalization of the Kingman subadditive
ergodic theorem that the author is aware of \cite{dooley2014sub}
only works for countable amenable groups $G$ which are strongly scale-invariant
in the sense of Nekrashevych and Pete \cite{nekrashevych2011scale}
(briefly, this implies that there exists an increasing Følner sequence
$(F_{n})$ such that each $F_{n}$ tiles $G$ and such that, in the
Cayley graph category, $\pi_{F_{n}}(G)$ is isomorphic to $G$) \emph{and}
only for Følner sequences which satisfy the Tempelman condition (which
do not exist for every amenable group), \emph{and} provided that an
additional technical side-condition is satisfied. 

\chapter{Fluctuation bounds}

\section{Introduction}

Consider the following version of the mean ergodic theorem for actions
of amenable groups:
\begin{thm*}
\emph{(Greenleaf }\cite{greenleaf1973ergodic}\emph{)} Let $L^{p}(S,\mu)$
be such that either $S$ is $\sigma$-finite and $1<p<\infty$ or
$\mu(S)<\infty$ and $p=1$, and let $x\in L^{p}(S,\mu)$. Let $G$
be a locally compact second countable amenable group with Haar measure
$dg$, let $G$ act continuously on $(S,\mu)$ by measure preserving
transformations, and let $(F_{n})$ be a Følner sequence of compact
subsets of $G$. Then $A_{n}x:=\frac{1}{|F_{n}|}\int_{F_{n}}\pi(g^{-1})x$
converges in $L^{p}$. 
\end{thm*}
Greenleaf proves this result by way of an abstract Banach space analogue
of the mean ergodic theorem which is simultaneously general enough
to deduce the mean ergodic theorem for an amenable group acting on
any reflexive Banach space or any $L^{1}(\mu)$ with $\mu$ a finite
measure. Central to Greenleaf's proof is a fixed point argument which
in particular does not give any effective convergence information
about the averages $A_{n}x$. 

Here our aim is to give an effective analogue of Greenleaf's theorem.
At the cost of some generality \textemdash{} here, we only consider
actions of amenable groups on \emph{uniformly convex} Banach spaces
\textemdash{} we obtain an explicit fluctuation bound for $(A_{n}x)$.

\section{Preliminaries}

We first fix some notation and terminology. 

A locally compact group $G$ will always come equipped with a Haar
measure, at least tacitly. In the countable discrete case this coincides
with the counting measure. Regardless of whether the group is discrete
or continuous, we will use the notations $dg$ and $\vert\cdot\vert$
interchangeably to refer to the Haar measure. 

A normed vector space $(\mathcal{B},\Vert\cdot\Vert)$ is said to
be uniformly convex if there exists a nondecreasing function $u(\varepsilon)$
such that for all $x,y\in\mathcal{B}$ with $\Vert x\Vert\leq\Vert y\Vert\leq1$
and $\Vert x-y\Vert\geq\varepsilon$, it follows that $\Vert\frac{1}{2}(x+y)\Vert<\Vert y\Vert-u(\varepsilon)$.
Such a function $u(\varepsilon)$ is then referred to as a \emph{modulus
of uniform convexity }for $\mathcal{B}$. 

In general, we say that a group $G$ acts on a normed vector space
$(\mathcal{B},\Vert\cdot\Vert)$ if there is a function $\pi(g)$
that returns an operator on $\mathcal{B}$ for every $g\in G$, $\pi(e)$
is the identity operator, and for all $g,h\in G$, $\pi(g)\pi(h)=\pi(gh)$.
Together these imply that $\pi(g)^{-1}=\pi(g^{-1})$. We say that
$G$ \emph{acts linearly on} $\mathcal{B}$ provided that in addition,
$\pi$ maps from $G$ to the space $\mathcal{L}(\mathcal{B},\mathcal{B})$
of linear operators on $\mathcal{B}$. 
Writing $\mathcal{L}_{1}(\mathcal{B},\mathcal{B})$ to indicate the
set of all linear operators from $\mathcal{B}$ to $\mathcal{B}$
with supremum norm $1$, another way to say that $G$ acts both linearly
and with unit norm on $\mathcal{B}$ is to say that $G$ acts on $\mathcal{B}$
via $\pi:G\rightarrow\mathcal{L}_{1}(\mathcal{B},\mathcal{B})$.\footnote{We remark that any group that acts via a representation $\pi:G\rightarrow\mathcal{L}(\mathcal{B},\mathcal{B})$
such that every $\pi(g)$ is \emph{nonexpansive} actually does so
via $\pi:G\rightarrow\mathcal{L}_{1}(\mathcal{B},\mathcal{B})$, by
the fact that $\pi(g^{-1})=\pi(g)^{-1}$ and the general fact about
linear operators that $\Vert T^{-1}\Vert\geq\Vert T\Vert^{-1}$. Nonexpansivity
is required for the proof of our main result.} Likewise, we say that a topological group $G$ \emph{acts continuously
on }$\mathcal{B}$ provided that for every $x\in\mathcal{B}$, if
$g\rightarrow e$ then $\Vert\pi(g)x-x\Vert\rightarrow0$. In other
words $g\mapsto\pi(g)x$ is continuous from $G$ to $\mathcal{B}$.
In the case where $G$ also acts linearly (resp. and with unit norm)
on $\mathcal{B}$, this is equivalent to requiring that $\pi:G\rightarrow\mathcal{L}(\mathcal{B},\mathcal{B})$
(resp. $\pi:G\rightarrow\mathcal{L}_{1}(\mathcal{B},\mathcal{B})$)
is continuous when $\mathcal{L}(\mathcal{B},\mathcal{B})$ is equipped
with the strong operator topology. 

Finally, we say that if $G$ is understood as a measurable space,
then $G$ \emph{acts strongly on $\mathcal{B}$} provided that for
every $x\in\mathcal{B}$, $g\mapsto\pi(g)x$ is strongly measurable
from $G$ to $\mathcal{B}$ (see Appendix A). In the case where $G$
also acts linearly (resp. and with unit norm) on $\mathcal{B}$, this
is equivalent to requiring that $\pi:G\rightarrow\mathcal{L}(\mathcal{B},\mathcal{B})$
(resp. $\pi:G\rightarrow\mathcal{L}_{1}(\mathcal{B},\mathcal{B})$)
is strongly measurable when $\mathcal{L}(\mathcal{B},\mathcal{B})$
is equipped with the strong operator topology. It is this very last
condition \textemdash{} $\pi:G\rightarrow\mathcal{L}_{1}(\mathcal{B},\mathcal{B})$
is strongly measurable when $\mathcal{L}(\mathcal{B},\mathcal{B})$
is equipped with the strong operator topology \textemdash{} that we
will actually use in our proof. To be briefer, we will say that $G$
acts strongly on $\mathcal{B}$ via the representation $\pi:G\rightarrow\mathcal{L}_{1}(\mathcal{B},\mathcal{B})$. 

For the convenience of the reader we recall some basic facts about
vector-valued integration. All of these can be found in, for example,
the recent textbook by Hytönen et al. \cite{hytonen2016analysis}.
\begin{prop}
(1) If $\int_{A}f(g)dg$ is either the Bochner or the Pettis integral,
then $\Vert\int_{A}f(g)dg\Vert\leq\int_{A}\Vert f(g)\Vert dg$.

(2) If $\int_{A}f(g)dg$ is either the Bochner or the Pettis integral,
and $T$ is a bounded linear operator, then $T(\int_{A}f(g)dg)=\int_{A}Tf(g)dg$.

(3) If $dg$ is $\sigma$-finite then Fubini's theorem holds for the
Bochner integral. 

(4) A strongly measurable function $f:G\rightarrow\mathcal{B}$ is
Bochner integrable iff $\int_{G}\Vert f(g)\Vert dg<\infty$, in other
words iff $\Vert f\Vert:G\rightarrow\mathbb{R}$ is integrable in
the Lebesgue sense.
\end{prop}

In what follows, therefore, every $\mathcal{B}$-valued integral is
understood to be a Bochner integral, and every $\mathbb{R}$-valued
integral is understood to be a Lebesgue integral. 

The following serves as our preferred characterization of amenability. 
\begin{defn}
(1) Let $G$ be a countable discrete group. A sequence $(F_{n})$
of finite subsets of $G$ is said to be a Følner sequence if for every
$\varepsilon>0$ and finite $K\subset G$, there exists an $N$ such
that for all $n\geq N$ and for all $k\in K$, $|F_{n}\Delta kF_{n}|<|F_{n}|\varepsilon$. 

(2) Let $G$ be a locally compact second countable (lcsc) group with
Haar measure $|\cdot|$. A sequence $(F_{n})$ of compact subsets
of $G$ is said to be a Følner sequence if for every $\varepsilon>0$
and compact $K\subset G$, there exists an $N$ such that for all
$n\geq N$, there exists a subset $K^{\prime}$ of $K$ with $|K^{\prime}|>(1-\varepsilon)|K|$
such that for all $k\in K^{\prime}$, $|F_{n}\Delta kF_{n}|<|F_{n}|\varepsilon$.

\end{defn}

\begin{rem*}
It has been observed, for instance, by Ornstein and Weiss \cite{ornstein1987entropy}
that (2) is one of several equivalent ``correct'' generalizations
of (1) to the lcsc setting. Note however, that we do not assume $(F_{n})$
is nested ($F_{i}\subset F_{i+1}$ for all $i\in\mathbb{N}$) or exhausts
$G$ ($\bigcup_{n\in\mathbb{N}}F_{n}=G$), nor do we assume, in the
lcsc case, that $G$ is unimodular. (Each of these is a common additional
technical assumption when working with amenable groups.) Conversely,
some authors use a version of (2) where the sets in $(F_{n})$ are
merely assumed to have finite volume, rather than compact; thanks
to the regularity of the Haar measure, our definition results in no
loss of generality. 
\end{rem*}
\begin{defn}
If $G$ is either a countable discrete or lcsc amenable group, and
has some distinguished Følner sequence $(F_{n})$, and acts on $\mathcal{B}$
via a representation $\pi:G\rightarrow\mathcal{L}_{1}(\mathcal{B},\mathcal{B})$,
then we define the $n$th \emph{ergodic average }operator as follows:
$A_{n}x:=\frac{1}{|F_{n}|}\int_{F_{n}}\pi(g^{-1})xdg$. 
\end{defn}

\begin{prop}
With the notation above, $\Vert A_{n}\Vert_{\mathcal{L}(\mathcal{B},\mathcal{B})}\leq1$.
\end{prop}

\begin{proof}
Observe that

\[
\Vert A_{n}x\Vert:=\left\Vert \frac{1}{|F_{n}|}\int_{F_{n}}\pi(g^{-1})xdg\right\Vert \leq\frac{1}{|F_{n}|}\int_{F_{n}}\Vert\pi(g^{-1})x\Vert dg\leq\frac{1}{|F_{n}|}\int_{F_{n}}\Vert x\Vert dg=\Vert x\Vert.
\]
\end{proof}
\begin{rem*}
To tie all this abstraction back to our original setting of interest,
we should note that in Appendix A, it is shown that if $G$ acts continuously
on $\mathcal{B}$, then $G$ acts strongly on $\mathcal{B}$. Consequently,
the ``concrete'' version of Greenleaf's mean ergodic theorem, where
$G$ acts continuously and by measure-preserving transformations on
a $\sigma$-finite measure space $(S,\mu)$, and $f\in L^{p}$ with
$p\in(1,\infty)$ (equivalently: the induced action of $G$ on $L^{p}(S,\mu)$
is a continuous action by linear isometries) so in particular $G$
acts via a unitary representation $\pi:G\rightarrow\mathcal{L}_{1}(\mathcal{B},\mathcal{B})$
which is continuous in the strong operator topology. It follows that
studying an ``abstract version'' where $\mathcal{B}$ is an arbitrary
uniformly convex Banach space and $G$ acts strongly on $\mathcal{B}$
via the representation $\pi:G\rightarrow\mathcal{L}_{1}(\mathcal{B},\mathcal{B})$
is, in fact, a bona fide generalization of the concrete version. 
\end{rem*}
A key piece of quantitative information for us will be how large $N$
has to be if $K$ is chosen to be an element of $(F_{n})$. This information
is encoded by the following type of modulus:
\begin{defn}
Let $G$ be an amenable group, either countable discrete or lcsc,
with Følner sequence $(F_{n})$. A \emph{Følner convergence modulus}
$\beta(n,\varepsilon)$ for $(F_{n})$ returns an integer $N$ such
that:

\begin{enumerate}
\item If $G$ is countable discrete, $(\forall m\geq N)(\forall g\in F_{n})\left[|F_{m}\Delta gF_{m}|<|F_{m}|\varepsilon\right]$.
\item If $G$ is lcsc, $(\forall m\geq N)(\exists F_{n}^{\prime}\subset F_{n})(\forall g\in F_{n}^{\prime})\left[|F_{n}\backslash F_{n}^{\prime}|<|F_{n}|\varepsilon\wedge|F_{m}\Delta gF_{m}|<|F_{m}|\varepsilon\right]$.
\end{enumerate}
\end{defn}

We remark that if $(F_{n})$ is an \emph{increasing} Følner sequence
(that is, $F_{n}\subset F_{m}$ for all $n\leq n$) then it follows
trivially that $\beta(n,\varepsilon)$ is a nondecreasing function
for any fixed $\varepsilon$. However, in what follows we do not always
assume that $(F_{n})$ is increasing. In some instances it is technically
convenient to assume that $\beta(n,\varepsilon)$ is non-decreasing;
in this case, we can upper bound $\beta(n,\varepsilon)$ using an
``envelope'' of the form $\tilde{\beta}(n,\varepsilon)=\max_{1\leq i\leq n}\beta(n,\varepsilon)$.
Hence, in any case we are free to assume that $\beta(n,\varepsilon)$
is non-decreasing in $n$ if necessary. 
\begin{example}
\emph{Computing some Følner convergence moduli. }\textcolor{red}{\emph{}}
\emph{}

\begin{enumerate}
\item Consider $\mathbb{Z}^{2}$ equipped with the Følner sequence composed
of the symmetric squares $[-m,m]^{2}$. If we shift such a square
by an element $(n_{1},n_{2})\in[-m,m]^{2}$, then the symmetric difference
between $[-m,m]^{2}$ and $(n_{1},n_{2})[-m,m]^{2}$ has cardinality
$2(2m+1)|n_{1}|+2(2m+1-|n_{1}|)|n_{2}|$. This quantity increases
with both $|n_{1}|$ and $|n_{2}|$. Suppose then that $(n_{1},n_{2})$
is taken from a 2-cube $[-n,n]^{2}$. Then the symmetric difference
is maximized when $n_{1}=n_{2}=n$ and 
\[
\frac{|[-m,m]^{2}\Delta(n,n)[-m,m]^{2}|}{|[-m,m]|^{2}}=\frac{4(2m+1)n-2n^{2}}{(2m+1)^{2}}<\frac{4n}{2m+1}
\]
Therefore if we pick $m\geq\frac{n}{2\varepsilon}$, it follows that
for all $(n_{1},n_{2})$ in the square $[-n,n]^{2}$, then $|[-m,m]^{2}\Delta(n_{1},n_{2})[-m,m]^{2}|<|[-m,m]^{2}|\varepsilon$.
Hence we can take $\beta(n,\varepsilon)=\lceil\frac{n}{2\varepsilon}\rceil$.
A similar computation for $d$-dimensional symmetric cubes in $\mathbb{Z}^{d}$
indicates that we can take $\beta(m,\varepsilon)\leq\lceil\frac{n}{2^{d-1}\varepsilon}\rceil.$

\item A slightly more interesting case is the solvable Baumslag-Solitar
group $BS(1,2)=\langle a,b\mid bab^{-1}=a^{2}\rangle$. We saw this
group in Example \ref{exa:Baumslag}, where we observed that it has
a Følner sequence of the form $(R_{2^{2k},k})$, where in general
$R_{m,n}$ denotes a rectangular subset of the form $\{b^{-n}a^{k}b^{j}\mid k\in[-m,m],j\in[0,2n]\}$.
We also observed that for all $g\in R_{m,n}$, 
\[
\frac{|R_{2^{2k},k}\Delta gR_{2^{2k},k}|}{|R_{2^{2k},k}|}\leq\frac{2n+2m}{2k+1}.
\]
Thus, if $g\in R_{2^{2j},j}$ (with $j\leq k$), we have 
\[
\frac{|R_{2^{2k},k}\Delta gR_{2^{2k},k}|}{|R_{2^{2k},k}|}\leq\frac{2\cdot2^{2j}+2j}{2k+1}.
\]
So given $\varepsilon>0$, in order for $|R_{2^{2k},k}\Delta gR_{2^{2k},k}|/|R_{2^{2k},k}|$
to be less than $\varepsilon$, it suffices to pick $k$ sufficiently
large that $(2\cdot2^{2j}+2j)/(2k+1)<\varepsilon$, in other words,
\[
\frac{2^{2j}+j}{\varepsilon}-\frac{1}{2}<k.
\]
Consequently, for the Følner sequence $(R_{2^{2k},k})$ on $BS(1,2)$,
we have that $\beta(j,\varepsilon)=\lceil\frac{2^{2j}+j}{\varepsilon}\rceil$
is a valid Følner convergence modulus.
\end{enumerate}
It is worth noting that under some reasonable assumptions, it is easy
to see that we can select a Følner sequence in such a way that $\beta(n,\varepsilon)$
can be chosen to be a computable function (for an appropriate restriction
on the domain of the second variable). The following argument has
essentially already been observed by previous authors \cite{cavaleri2017computability,cavaleri2018folner,moriakov2018effective}
working with slightly different objects, but we include it for completeness.
\end{example}

\begin{prop}
Let $G$ be a countable discrete finitely generated amenable group
with the solvable word property. Fix $k\in\mathbb{N}$. Then $G$
has a Følner sequence $(F_{n})$ such that $\beta(n,k^{-1})=\max\{n+1,k\}$
is a Følner convergence modulus for $(F_{n})$. Moreover $(F_{n})$
can be chosen in a computable fashion.
\end{prop}

\begin{proof}
Fix a computable enumeration of the finite subsets of $G$. The solvable
word property ensures that we can do this, and also that the cardinality
of $F\Delta gF$ can always be computed for any $g\in G$ and finite
set $F$. So, take $F_{1}$ to be an arbitrary finite set. Given $F_{n-1}$,
take $F_{n}$ to be the least (with respect to the enumeration) finite
subset of $G$ containing $F_{n-1}$, such that for all $g\in F_{n-1}$,
$|F_{n}\Delta gF_{n}|<|F_{n}|/n$. Such an $F_{n}$ exists since $G$
is amenable. This is indeed a Følner sequence: for a fixed $g$, we
see that $|F_{n}\Delta gF_{n}|<|F_{n}|/n$ for all $n$ greater than
the first $m$ such that $g\in F_{m}$, hence $|F_{n}\Delta gF_{n}|/|F_{n}|\rightarrow0$.
Moreover, we see that if $m\geq\max\{n+1,k\}$, then 
\[
(\forall g\in F_{n})\qquad|F_{m}\Delta gF_{m}|<|F_{m}|/m\leq|F_{m}|/k.
\]
\end{proof}
\begin{rem*}
The previous proposition is not sharp. It has been shown that there
are groups \emph{without} the solvable word property which nonetheless
have computable Følner sequences with computable convergence behaviour
\cite{cavaleri2018folner}. (The cited paper uses a different explicit
modulus of convergence for Følner sequences than the present paper,
although the argument carries over to our setting without modification.) 
\end{rem*}

\section{The Main Theorem}

Frequently in ergodic theory, one argues that if $K\gg N$, then $A_{K}A_{N}x\approx A_{K}x$.
The following lemma makes this precise in terms of the modulus $\beta$. 
\begin{lem}
Let $(\mathcal{B},\Vert\cdot\Vert)$ be a normed vector space. Let
$G$ be a lcsc amenable group with Følner sequence $(F_{n})$, and
let $G$ act strongly on $\mathcal{B}$ via the representation $\pi:G\rightarrow\mathcal{L}_{1}(\mathcal{B},\mathcal{B})$.
Fix $N\in\mathbb{N}$ and $\eta>0$. Let $\beta$ be the Følner convergence
modulus and suppose $K\geq\beta(N,\eta)$. Then for any $x\in\mathcal{B}$,
$\Vert A_{K}x-A_{K}A_{N}x\Vert<3\eta\Vert x\Vert$. (If $G$ is countable
discrete, strong measurability is trivially satisfied, and we have
the sharper estimate $\Vert A_{K}x-A_{K}A_{N}x\Vert<\eta\Vert x\Vert$.)
\end{lem}

\begin{proof}
From the definition of Følner convergence modulus, we know that there
exists an $F_{N}^{\prime}\subset F_{N}$ such that $|F_{N}^{\prime}|<(1-\eta)|F_{N}|$
and such that for all $h\in F_{N}^{\prime}$, $|F_{K}\Delta hF_{K}|<|F_{K}|\eta$.
Now perform the following computation (justification for each step
addressed below): 
\[
\Vert A_{K}x-A_{K}A_{N}x\Vert:=\left\Vert \frac{1}{|F_{K}|}\int_{F_{K}}\pi(g^{-1})xdg-\frac{1}{|F_{K}|}\int_{F_{K}}\pi(g^{-1})\left(\frac{1}{|F_{N}|}\int_{F_{N}}\pi(h^{-1})xdh\right)dg\right\Vert 
\]
\begin{align*}
 & =\left\Vert \frac{1}{|F_{K}|}\int_{F_{K}}\pi(g^{-1})xdg-\frac{1}{|F_{K}||F_{N}|}\int_{F_{K}}\left(\int_{F_{N}}\pi(g^{-1})(\pi(h^{-1})x)dh\right)dg\right\Vert \\
 & =\left\Vert \frac{1}{|F_{K}|}\int_{F_{K}}\pi(g^{-1})xdg-\frac{1}{|F_{K}||F_{N}|}\int_{F_{K}}\left(\int_{F_{N}}\pi((hg)^{-1})xdh\right)dg\right\Vert \\
 & =\left\Vert \frac{1}{|F_{N}|}\int_{F_{N}}\left(\frac{1}{|F_{K}|}\int_{F_{K}}\pi(g^{-1})xdg\right)dh-\frac{1}{|F_{N}||F_{K}|}\int_{F_{N}}\left(\int_{F_{K}}\pi((hg)^{-1})xdg\right)dh\right\Vert \\
 & \leq\frac{1}{|F_{N}|}\int_{F_{N}}\left\Vert \frac{1}{|F_{K}|}\int_{F_{K}}\pi(gg^{-1})xdg-\frac{1}{|F_{K}|}\int_{F_{K}}\pi((hg)^{-1})xdg\right\Vert dh\\
 & =\frac{1}{|F_{N}|}\int_{F_{N}}\left\Vert \frac{1}{|F_{K}|}\int_{F_{K}}\pi(g^{-1})x-\frac{1}{|F_{K}|}\int_{hF_{K}}\pi(g^{-1})xdg\right\Vert dh\\
 & \leq\frac{1}{|F_{N}|}\int_{F_{N}}\left(\frac{1}{|F_{K}|}\int_{F_{K}\Delta hF_{K}}\Vert\pi(g^{-1})x\Vert dg\right)dh\\
 & \leq\frac{1}{|F_{N}|}\int_{F_{N}}\left(\frac{1}{|F_{K}|}\int_{F_{K}\Delta hF_{K}}\Vert x\Vert dg\right)dh\\
 & =\frac{1}{|F_{N}|}\int_{F_{N}}\frac{1}{|F_{K}|}\left(|F_{K}\Delta hF_{K}|\Vert x\Vert\right)dh\\
 & <\frac{1}{|F_{N}|}\left[\int_{F_{N}^{\prime}}\eta\Vert x\Vert dh+\int_{F_{N}\backslash F_{N}^{\prime}}\left(\frac{1}{|F_{K}|}|F_{K}\Delta hF_{K}|\Vert x\Vert\right)dh\right]\\
 & \leq\eta\Vert x\Vert+\frac{1}{|F_{N}|}\int_{F_{N}\backslash F_{N}^{\prime}}\left(2\Vert x\Vert\right)dh\leq3\eta\Vert x\Vert.
\end{align*}
If $G$ is countable discrete, we instead assume that for all $h\in F_{N}$
(rather than $F_{N}^{\prime}$), $|F_{K}\Delta hF_{K}|<|F_{K}|\eta$.
Therefore, the penultimate line reduces to $\frac{1}{|F_{N}|}\int_{F_{N}}\eta\Vert x\Vert dh$,
and the last line reduces to $\eta\Vert x\Vert$. 

Finally let's discuss which properties of the Bochner integral we
had to use. If, for each $g$, $\pi(g)$ is a bounded linear operator,
then indeed it follows that $\pi(g^{-1})\left(\int\pi(h^{-1})xdh\right)=\int(\pi(g^{-1})\pi(h^{-1})xdh$.
If Fubini's theorem holds, then indeed $\int_{F_{K}}\int_{F_{N}}\pi((hg)^{-1})xdhdg=\int_{F_{N}}\int_{F_{K}}\pi((hg)^{-1})xdhdg$.
Here, Fubini's theorem is guaranteed by strong measurability, together
with the continuity of group multiplication (!) \textemdash{} see
Appendix A. Lastly, we repeatedly invoked the fact that $\Vert\int_{A}f(g)dg\Vert\leq\int_{A}\Vert f(g)\Vert dg$.
It's worth noting that retreating to the case where $G$ is countable,
only the first fact (that $G$ acts by bounded linear operators) is
needed as an assumption, as the latter two properties hold trivially
for finite averages.
\end{proof}
\begin{rem*}
It is possible to generalize this argument to the case where the action
of $G$ is ``power bounded'' in the sense that there is some uniform
constant $C$ such that for ($dg$-almost) all $g\in G$, $\Vert\pi(g)\Vert\leq C$.
However the argument for our main theorem necessitates setting $C=1$.
\end{rem*}
The following argument is a generalization of proof of Garrett Birkhoff
\cite{birkhoff1939} to the amenable setting. The statement of the
theorem is weaker than results which are already in Greenleaf's article
\cite{greenleaf1973ergodic}, but we include the argument for several
reasons. One is that it is very short; another is that we will ultimately
derive a bound on $\varepsilon$-fluctuations via a modification of
this proof; and finally, the proof indicates additional information
about the limiting behaviour of the norm of $A_{n}x$, namely that
$\lim_{n}\Vert A_{n}x\Vert=\inf_{n}\Vert A_{n}x\Vert$.
\begin{thm}
Let $G$ be a locally compact, second countable amenable group with
compact Følner sequence $(F_{n})$, and let $\mathcal{B}$ a uniformly
convex Banach space such that $G$ acts strongly on $\mathcal{B}$
via the representation $\pi:G\rightarrow\mathcal{L}_{1}(\mathcal{B},\mathcal{B})$.
Then for every $x\in\mathcal{B}$, the sequence of averages $(A_{n}x)$
converges in norm $\Vert\cdot\Vert_{\mathcal{B}}$.
\end{thm}

\begin{proof}
Without loss of generality, we assume $\Vert x\Vert\leq1$. Define
$L:=\inf_{n}\Vert A_{n}x\Vert$. Fix an $\varepsilon_{0}$, and let
$N$ be some index such that $\Vert A_{N}x\Vert<L+\varepsilon_{0}$.
Let $u$ denote the modulus of uniform convexity. Suppose that $M>N$
is an index such that $\Vert A_{N}x-A_{M}x\Vert>\delta$. (If no such
$\delta$ exists then this means that after $\beta(N,\eta)$, the
sequence has converged to within $\delta$.) Then this implies that
\[
\left\Vert \frac{1}{2}(A_{N}x+A_{M}x)\right\Vert \leq\max\{\Vert A_{N}x\Vert,\Vert A_{M}x\Vert\}-u(\varepsilon).
\]
The idea is that if we know $M\gg N$, then $\Vert A_{M}x\Vert\approx\Vert A_{M}A_{N}x\Vert\leq\Vert A_{N}x\Vert$.
Therefore, fix a Følner convergence modulus $\beta(n,\varepsilon)$
for $(F_{n})$, and suppose $M\geq\beta(N,\eta/(3\Vert x\Vert))$.
It follows from the lemma that $\Vert A_{M}x-A_{M}A_{N}x\Vert<\eta$,
and therefore 
\[
\left\Vert \frac{1}{2}(A_{N}x+A_{M}x)\right\Vert <\max\{\Vert A_{N}x\Vert,\Vert A_{M}A_{N}x\Vert+\eta\}-u(\delta).
\]
But $\Vert A_{M}A_{N}x\Vert\leq\Vert A_{N}x\Vert$, so this implies
\[
\left\Vert \frac{1}{2}(A_{N}x+A_{M}x)\right\Vert <\Vert A_{N}x\Vert+\eta-u(\delta).
\]
In turn, we know that $\Vert A_{N}x\Vert<L+\varepsilon_{0}$, and
by assumption $\Vert A_{N}x-A_{M}x\Vert<\delta$, so 
\[
\left\Vert \frac{1}{2}(A_{N}x+A_{M}x)\right\Vert <L+\varepsilon_{0}+\eta-u(\delta).
\]
In fact, it follows that $\Vert\frac{1}{2}A_{K}(A_{N}x+A_{M}x)\Vert<L+\varepsilon_{0}+\eta-u(\delta)$
also, for any index $K$. Now, choosing $K\geq\max\{\beta(N,\eta/(3\Vert x\Vert)),\beta(M,\eta/(3\Vert x\Vert))\}$,
we have that both $\Vert A_{K}x-A_{K}A_{N}x\Vert<\eta$ and $\Vert A_{K}x-A_{K}A_{M}x\Vert<\eta$.
Thus, 
\begin{align*}
\Vert A_{K}x\Vert & =\left\Vert \frac{1}{2}(A_{K}x-A_{K}A_{N}x)+\frac{1}{2}(A_{K}x-A_{K}A_{M}x)+\frac{1}{2}(A_{K}A_{N}+A_{K}A_{M})\right\Vert \\
 & \leq\eta+\left\Vert \frac{1}{2}A_{K}(A_{N}x+A_{M}x)\right\Vert \\
 & <2\eta+L+\varepsilon_{0}-u(\delta).
\end{align*}
Since $\eta$ can be chosen to be arbitrarily small provided that
$K$ (and $M$) is sufficiently large, we see that $\limsup\Vert A_{K}x\Vert\leq L+\varepsilon_{0}-u(\delta)$.
But since our choice of $\varepsilon_{0}$ was arbitrary, and $u(\delta)<\varepsilon_{0}+\eta$,
it follows that in fact $\limsup_{K}\Vert A_{K}x\Vert\leq m=\inf_{n}\Vert A_{n}x\Vert$.
Moreover this implies that $(A_{n}x)$ converges in norm. For if this
were not the case, then we could find some $\delta_{0}$ such that
$\Vert A_{n}x-A_{m}x\Vert>\delta_{0}$ infinitely often. Picking $\eta$
and $\varepsilon_{0}$ small enough that $2\eta+\varepsilon_{0}<u(\delta_{0})$,
and picking both $n$ and $m$ sufficiently large that $\Vert A_{n}x\Vert,\Vert A_{m}x\Vert<L+\varepsilon_{0}$,
the above computation shows that for $k$ larger that $\beta(m,\eta)$
and $\beta(n,\eta)$ we have that $\Vert A_{k}x\Vert<2\eta+L+\varepsilon_{0}-u(\delta_{0})<L$,
which contradicts the definition of $L$.
\end{proof}
We now proceed to deriving a quantitative analogue of this result.
To do so we introduce the following notion.
\begin{defn}
Let $G$ be a countable discrete or lcsc amenable group and $(F_{n})$
a Følner sequence. Let $\lambda\in\mathbb{N}$ and $\varepsilon>0$.
We say that $(F_{n})$ is a \emph{$(\lambda,\varepsilon)$-fast} Følner
sequence if

\begin{enumerate}
\item For $G$ countable and discrete, it holds that for all $n\in\mathbb{N}$
that for all $k\leq n$ and for all $m\geq k+\lambda$, for all $g\in F_{k}$,
$|F_{m}\Delta gF_{m}|/|F_{m}|<\varepsilon$.
\item For $G$ lcsc, it holds that for all $n\in\mathbb{N}$ that for all
$k\leq n$ and for all $m\geq k+\lambda$, there exists a set $F_{k}^{\prime}\subset F_{k}$
such that $|F_{k}\backslash F_{k}^{\prime}|<|F_{k}|\varepsilon$,
so that for all $g\in F_{k}^{\prime}$, $|F_{m}\Delta gF_{m}|/|F_{m}|<\varepsilon$.
\end{enumerate}
\end{defn}

It is clear that any Følner sequence can be refined into a $(\lambda,\varepsilon)$-fast
Følner sequence. Less clear is the relationship between a Følner sequence
being fast and the property of being \emph{tempered} which is used
in Lindenstrauss's pointwise ergodic theorem, although they are somewhat
similar in spirit.
\begin{prop}
Given $\lambda\in\mathbb{N}$ and $\varepsilon>0$, any Følner sequence
can be refined into a $(\lambda,\varepsilon)$-fast Følner sequence. 
\end{prop}

\begin{proof}
It suffices to produce a $(1,\varepsilon)$-fast refinement. For simplicity,
we only state the argument for the case where $G$ is countable and
discrete. 

Suppose we have already selected the first $j$ Følner sets in our
refinement $F_{n_{1}},\ldots,F_{n_{j}}$. Then, take $F_{n_{j+1}}$
to be the next element of the sequence $(F_{n})$ after $n_{j}$ such
that, for all $g\in\bigcup_{i=1}^{j}F_{n_{i}}$, $|F_{n_{j+1}}\Delta gF_{n_{j+1}}|/|F_{n_{j+1}}|<\varepsilon$.
Such a term exists since $(F_{n})$ is a Følner sequence. 
\end{proof}

Let's now count the $\varepsilon$-fluctuations. We first do so ``at
distance $\beta$'' (see Chapter 1), and then recover a global bound
in the case where the Følner sequence is fast. The only really non-explicit
of the proof of the preceding theorem was the step where we used the
fact that an infimum of a real sequence exists. In contrast to Kohlenbach
and Leustean, who perform a functional interpretation on the classical
statement asserting the existence of an infimum, we will just use
the crude fact that the infimum is nonnegative.

Fix a non-decreasing Følner convergence modulus $\beta$ for $(F_{n})$.
Suppose that $\Vert A_{n_{0}}x-A_{n_{1}}x\Vert\geq\varepsilon$. Moreover,
we suppose that $n_{1}\geq\beta(n_{0},\eta/3\Vert x\Vert)$. 

{} Then the computation from the previous proof shows that 
\[
\left\Vert \frac{1}{2}(A_{n_{0}}x+A_{n_{1}}x)\right\Vert <\Vert A_{n_{0}}x\Vert+\eta-u(\varepsilon).
\]
More generally, if $\Vert A_{n_{i}}x-A_{n_{i+1}}x\Vert\geq\varepsilon$
with $n_{i+1}\geq\beta(n_{i},\eta/3\Vert x\Vert)$, it follows that
\[
\left\Vert \frac{1}{2}(A_{n_{i}}x+A_{n_{i+1}}x)\right\Vert <\Vert A_{i}x\Vert+\eta-u(\delta).
\]
Now, choosing $k\geq\max\{\beta(n_{i+1},\eta/(3\Vert x\Vert)),\beta(n_{i},\eta/(3\Vert x\Vert))\}$,
we have that both $\Vert A_{k}x-A_{k}A_{n_{i}}x\Vert<\eta$ and $\Vert A_{k}x-A_{k}A_{n_{i+1}}x\Vert<\eta$.
Thus, 
\begin{align*}
\Vert A_{k}x\Vert & =\left\Vert \frac{1}{2}(A_{k}x-A_{k}A_{n_{i}}x)+\frac{1}{2}(A_{k}x-A_{k}A_{n_{i+1}}x)+\frac{1}{2}(A_{k}A_{n_{i}}x+A_{k}A_{n_{i+1}}x)\right\Vert \\
 & \leq\eta+\left\Vert \frac{1}{2}A_{k}(A_{n_{i}}x+A_{n_{i+1}}x)\right\Vert \\
 & <2\eta+\Vert A_{n_{i}}x\Vert-u(\varepsilon).
\end{align*}
Therefore let $n_{i+2}$ equal the least index greater than $\max\{\beta(n_{i+1},\eta/(3\Vert x\Vert)),\beta(n_{i},\eta/(3\Vert x\Vert))\}$
(and therefore greater than $\beta(n_{i+1},\eta/(3\Vert x\Vert))$,
since $\beta$ is non-decreasing in $n$) such that $\Vert A_{n_{i+1}}x-A_{n_{i+2}}x\Vert\geq\varepsilon$.
The previous calculation shows that $\Vert A_{n_{i+2}}x\Vert<\Vert A_{n_{i}}x\Vert+2\eta-u(\varepsilon)$.
More generally, we have that 
\[
\Vert A_{n_{i}}x\Vert<\Vert A_{n_{0}}x\Vert-\frac{i}{2}\left(u(\varepsilon)-2\eta\right)\quad i\text{ even}
\]
\[
\Vert A_{n_{i}}x\Vert<\Vert A_{n_{1}}x\Vert-\frac{i-1}{2}(u(\varepsilon)-2\eta)\quad i\text{ odd}
\]
So simply from the fact that $\Vert A_{n_{i}}x\Vert\geq0$, these
expressions derive a contradiction on the least $i$ such that 
\[
\max\left\{ \Vert A_{n_{0}}x\Vert,\Vert A_{n_{1}}x\Vert\right\} <\frac{i-1}{2}(u(\varepsilon)-2\eta)
\]
since this would imply that $\Vert A_{n_{i}}(x)\Vert<0$. That is,
the contradiction implies that the $n_{i}$th epsilon fluctuation
could not have occurred. We have no a priori information on the norms
of $\Vert A_{n_{0}}x\Vert$ and $\Vert A_{n_{1}}x\Vert$, except that
both are at most $\Vert x\Vert$. Therefore, we have the following
uniform bound: 
\[
i\leq\left\lfloor \frac{2\Vert x\Vert}{u(\varepsilon)-2\eta}+1\right\rfloor 
\]
where $i$ tracks the indices of the subsequence along which $\varepsilon$-fluctuations
occur. This is actually \emph{one more} than the number of $\varepsilon$-fluctuations,
so instead we have that the number of $\varepsilon$-fluctuations
is bounded by $\left\lfloor \frac{2\Vert x\Vert}{u(\varepsilon)-2\eta}\right\rfloor $.

If we happen to have any lower bound on the infimum of $\Vert A_{n}x\Vert$,
we can sharpen the previous calculation. Instead of using the fact
that $\Vert A_{n}x\Vert\geq0$, we use the fact that $\Vert A_{n}x\Vert\geq L$
for some $L$. To wit, if $i$ is large enough that 
\[
\Vert x\Vert<\frac{i-1}{2}(u(\varepsilon)-2\eta)+L
\]
Then this would imply that $\Vert A_{n_{i}}x\Vert<L$, a contradiction.
Therefore we have the bound 
\[
i\leq\left\lfloor \frac{2(\Vert x\Vert-L)}{u(\varepsilon)-2\eta}+1\right\rfloor 
\]
and so the number of $\varepsilon$-fluctuations is bounded by $\left\lfloor \frac{2(\Vert x\Vert-L)}{u(\varepsilon)-2\eta}\right\rfloor $.

To summarize, we have shown that:
\begin{thm}
Let $\mathcal{B}$ be a uniformly convex Banach space with modulus
$u$. Fix $\varepsilon>0$ and $x\in\mathcal{B}$ with $\Vert x\Vert\leq1$.
Pick some $\eta<\frac{1}{2}u(\varepsilon)$. Then if $G\curvearrowright\mathcal{B}$
with Følner sequence $(F_{n})$, the sequence $(A_{n}x)$ has at most
$\left\lfloor \frac{2\Vert x\Vert}{u(\varepsilon)-2\eta}\right\rfloor $
$\varepsilon$-fluctuations at distance $\beta(n,\eta/3\Vert x\Vert)$.
If we know that $\inf\Vert A_{n}x\Vert\geq L$, then we can sharpen
the bound to $\left\lfloor \frac{2(\Vert x\Vert-L)}{u(\varepsilon)-2\eta}\right\rfloor $.
\end{thm}

\begin{cor}
In the above setting, suppose that $(F_{n})$ is $(\lambda,\eta/3\Vert x\Vert)$-fast.
Then the sequence $(A_{n}x)$ has at most $\lambda\cdot\left\lfloor \frac{2\Vert x\Vert}{u(\varepsilon)-2\eta}\right\rfloor +\lambda$
$\varepsilon$-fluctuations.
\end{cor}

\begin{proof}
We know from the theorem that there are at most $\left\lfloor \frac{2\Vert x\Vert}{u(\varepsilon)-2\eta}\right\rfloor $
$\varepsilon$-fluctuations at distance $\lambda$. This leaves the
possibility that there are some $\varepsilon$-fluctuations in the
$\left\lfloor \frac{2\Vert x\Vert}{u(\varepsilon)-2\eta}\right\rfloor $
many gaps of width $\lambda$, and also that there are some $\varepsilon$-fluctuations
in between the last possible index $n_{i}$ given by the previous
theorem, and the index $n_{i+1}$ at which contradiction is achieved.
This end last interval is at most $\lambda$ wide as well.
\end{proof}

\section{Discussion}

Our proof was carried out in the setting where the acted upon space
was assumed to be uniformly convex, and indeed our bound on the number
of fluctuations explicitly depends on the modulus of uniform convexity.
Nonetheless, it is natural to ask whether an analogous result might
be obtained for a more general class of acted upon spaces.

However, it has already been observed, in the case where $G=\mathbb{Z}$,
that there exists a separable, reflexive, and strictly convex Banach
space $\mathcal{B}$ such that for every $N$ and $\varepsilon>0$,
there exists an $x\in\mathcal{B}$ such that $(A_{n}x)$ has at least
$N$ $\varepsilon$-fluctuations \cite{avigad2015oscillation}. This
counterexample applies equally to bounds on the rate of metastability.

\textcolor{red}{{} }However, this counterexample does not directly eliminate
the possibility of a fluctuation bound for $\mathcal{B}=L^{1}(X,\mu)$,
so the question of a ``quantitative $L^{1}$ mean ergodic theorem
for amenable groups'' remains unresolved. 

What about the choice of acting group? Our assumptions on $G$ (amenable
and countable discrete or lcsc) where selected because this is the
most general class of groups which have Følner sequences. Our argument
depends essentially on Følner sequences; indeed, proofs of ergodic
theorems for actions of non-amenable groups have a qualitatively different
structure. Remarkably, there are certain classes of non-amenable groups
whose associated ergodic theorems have much stronger convergence behaviour
than the classical ($G=\mathbb{Z}$) setting; for recent progress
on quantitative ergodic theorems in the non-amenable setting, we refer
the reader to the book and survey article of Gorodnik and Nevo \cite{gorodnik2009ergodic,gorodnik2015quantitative}.

We should also mention quantitative bounds for \emph{pointwise} ergodic
theorems. For $G=\mathbb{Z}$ such results go as far back as Bishop's
upcrossing inequality. Inequalities of this type have also been found
for $\mathbb{Z}^{d}$ by Kalikow and Weiss (for $F_{n}=[-n,n]^{d}$)
\cite{kalikow1999fluctuations}; more recently Moriakov has modified
the Kalikow and Weiss argument to give an upcrossing inequality for
symmetric ball averages in groups of polynomial growth \cite{moriakov2018fluctuations}.
Presently it is unknown whether similar results hold for any larger
class of amenable groups. 

For both norm and pointwise convergence of ergodic averages, it is
sometimes possible to deduce convergence behavior which is stronger
than $\varepsilon$-fluctuations/upcrossings but weaker than an explicit
rate of convergence, namely that a sequence is \emph{bounded in total
variation} in a uniform fashion; these results are called \emph{variational
inequalities}. Jones et al. have succeeded in proving numerous variational
inequalities, both for norm and pointwise convergence, for a large
class of Følner sequences in $\mathbb{Z}$ and $\mathbb{Z}^{d}$ \cite{jones1998oscillation,jones2003oscillation}.
However, their methods, which rely on a martingale comparison and
a Calderón-Zygmund decomposition, exploit numerous incidental geometric
properties of $\mathbb{Z}^{d}$ which do not hold for many other groups.
It would be interesting to determine which other groups enjoy similar
variational inequalities.

\chapter*{Appendix A: Bochner integration}

Consider some measure space $(X,\mu)$ with some function $f:X\rightarrow\mathcal{B}$,
with $\mathcal{B}$ a Banach space. What would it look like to integrate
$f$?

One approach is to start with simple functions. In this setting, an
indicator function $\chi_{A}$ is real-valued as usual, but the ``scalar
coefficients'' are replaced by values in the Banach space. Thus $f(x)$
is a simple function if it is of the form $\sum_{i=1}^{N}1_{A_{i}}(x)b_{i}$
with $1_{A_{i}}(x)$ an indicator function for $A_{i}\subset X$ and
$b_{i}\in\mathcal{B}$. We then say that a function $f$ is \emph{strongly
measurable }if it a pointwise limit of simple functions, i.e. if there
exists a sequence $f_{n}$ of simple functions such that for every
$x\in X$, $\Vert f(x)-f_{n}(x)\Vert_{\mathcal{B}}\rightarrow0$. 

In general strong measurability is hard to come by. A more general
notion is \emph{weak measurability}: we say that $f:X\rightarrow\mathcal{B}$
is \emph{weakly measurable} if for every $b^{*}\in\mathcal{B}^{*}$,
the function $b^{*}\circ f:X\rightarrow\mathbb{R}$ is measurable
(in the ordinary sense as a function from $(X,\mu)$ to $\mathbb{R}$
with the Borel $\sigma$-algebra), which is a bit more manageable.
The following classical result indicates when weak measurability implies
strong measurability.
\begin{prop}
\emph{(Pettis measurability theorem) }Let $(X,\mu)$ be a measure
space and $\mathcal{B}$ a Banach space. For a function $f:X\rightarrow\mathcal{B}$
the following are equivalent:

\begin{enumerate}
\item $f$ is strongly measurable.
\item f is weakly measurable, and $f(X)$ is separable in $\mathcal{B}$. 
\end{enumerate}
\end{prop}

Here are some easy consequences.
\begin{prop}
If the measure space $(X,\mu)$ is also a separable topological space,
and $f:X\rightarrow\mathcal{B}$ is continuous, then $f$ is strongly
measurable.
\end{prop}

\begin{proof}
Observe that for any $b^{*}\in\mathcal{B}^{*}$, $b^{*}\circ f$ is
a composition of continuous functions, and is therefore continuous.
Hence $f$ is weakly measurable. Moreover, it holds that the continuous
image of a separable space is separable.
\end{proof}
\begin{prop}
If $(Y,\nu)$ is also a topological space, $\phi:(Y,\nu)\rightarrow(X,\mu)$
is continuous, and $f:(X,\mu)\rightarrow\mathcal{B}$ is strongly
measurable, then $f\circ\phi$ is strongly measurable.
\end{prop}

\begin{proof}
By hypothesis, $f$ is also weakly measurable, so for any $b^{*}\in\mathcal{B}^{*}$,
we have that $b^{*}\circ f$ is continuous. Therefore $b^{*}\circ(f\circ\phi)=(b^{*}\circ f)\circ\phi$
is continuous, and thus $f\circ\phi$ is weakly measurable. Now, let
$A=\phi(Y)$ be the image of $\phi$ in $X$. Note that $(f\circ\phi)(Y)=f(A)$.
Since $f(X)$ is separable in $\mathcal{B}$, it follows that $f(A)$
is also separable in $\mathcal{B}$ since it is contained in $f(X)$. 
\end{proof}
From this, we deduce the following fact which is important for our
purposes:
\begin{prop}
Suppose that $G$ acts strongly measurably on $\mathcal{B}$ (that
is, for every $x\in\mathcal{B}$, $g\mapsto\pi(g)x$ is strongly measurable).
Then for each $A\subset G$ with $\mu(A)<\infty$, we have that $1_{A}\pi(g)x$
is Bochner integrable for each $x\in\mathcal{B}$, i.e. $\int_{A}\Vert\pi(g)x\Vert d\mu(g)<\infty$.
Moreover, $\int_{A\times B}\Vert\pi(gh)x\Vert d\mu(g)\times d\mu(h)<\infty$,
and in particular $1_{A\times B}\pi(gh)x$ is strongly measurable
from $G\times G$ to $\mathcal{B}$. 
\end{prop}

\begin{proof}
(1) Since $\pi(\cdot)x$ is a strongly measurable function from $G$
to $\mathcal{B}$, it suffices to observe that 
\[
\int_{A}\Vert\pi(g)x\Vert d\mu(g)\leq\int_{A}\Vert x\Vert d\mu(g)<\infty.
\]
(2) Since $G$ is a topological group, we know that group multiplication
is continuous. Therefore $(g,h)\mapsto\pi(gh)x$ is strongly measurable,
since it is a composition of the continuous multiplication function
and the strongly measurable function $\pi(\cdot)x$. We also have
Bochner integrability because again, 
\[
\int_{A\times B}\Vert\pi(gh)x\Vert d\mu(g)\times d\mu(h)\leq\int_{A\times B}\Vert x\Vert d\mu(g)\times d\mu(h)<\infty.
\]
\end{proof}

\chapter*{Appendix B: Two logical addenda}

\section*{B.1 Effective learnability versus fluctuations at distance $\beta$}

The primary proof-theoretic reference on fluctuations at distance
$\beta$ is \emph{Fluctuations, effective learnability, and metastability
in analysis} by Kohlenbach and Safarik \cite{kohlenbach2014fluctuations}.
(This is indicated, for example, in Towsner's paper \emph{Nonstandard
analysis gives bounds on jumps} \cite{towsner2017nonstandard}, which
addresses fluctuations at distance $\beta$ from a model-theoretic
perspective; see also B.2 below.) However, this work does not actually
directly refer to fluctuations at distance $\beta$ anywhere! We therefore
spend a few words explaining how this paper \emph{is} actually talking
about fluctuations at distance $\beta$, albeit couched in a markedly
different vocabulary. 

Here, the term \emph{Cauchy statement} refers to a statement of the
form 
\[
\varphi(k):=\exists n\in\mathbb{N}\forall j\in\mathbb{N}(j\geq n\rightarrow d(x_{j},x_{n})<2^{-k}).
\]
Structurally, a Cauchy statement (with $k$ fixed in advance) has
the form 
\[
\exists n\in\mathbb{N}\forall j\in\mathbb{N}\varphi_{0}(j,n,(x_{n}))
\]
and is \emph{monotone} in $n$, i.e.
\[
\forall n\in\mathbb{N}\forall n^{\prime}\geq n\forall j\in\mathbb{N}(\varphi_{0}(j,n,(x_{n})\rightarrow\varphi_{0}(j,n^{\prime}(x_{n}))
\]
 if we think of the particular sequence $(x_{n})$ as being a parameter.
(Monotonicity just encodes the fact that if we assert that a sequence
has converged to within $\varepsilon$ after $N$, then the same holds
for $N^{\prime}\geq N$.) Thus, Cauchy statements are a special case
of statements of the above form, (namely monotone $\Sigma_{2}^{0}$
formulas with a single sequence parameter). 

Now, suppose that we attempt to ``learn'' the limit of a real-valued
sequence up to some error term (say $2^{-k}$), in the following fashion. 
\begin{enumerate}
\item Before looking at the sequence, we guess that the limit is in $(x_{0}-2^{-k},x_{0}+2^{-k})$.
For the sake of notational consistency, put $c_{0}:=0$ to denote
the fact that $c_{0}$ is our initial guess for an index of the sequence
$(x_{n})$ such that all terms of higher index stay within the $2^{-k}$-ball
around the term with index $c_{0}$. 
\item After looking at the first $j-1$ terms of the sequence, we have a
current guess $c_{i}$. If $x_{j}\in(c_{i}-2^{-k},c_{i}+2^{-k})$,
then we keep our guess the same, and keep $c_{i}$. Otherwise, put
$c_{i+1}=j$. 
\end{enumerate}
Evidently, this procedure will terminate (in the sense that the guess
$c_{n}$ is modified only a finite number of times) iff the Cauchy
statement $\varphi(k)$ is true about the sequence $(x_{n})$! More
generally, this procedure will terminate for every $k$ iff $(x_{n})$
is a Cauchy sequence. However, more directly, the learning procedure
we have just described ``changes its mind'' every time it detects
an $\varepsilon$-fluctuation; asserting that the learning procedure
will terminate after a fixed number of steps is thus equivalent to
asserting that if we search along the sequence $(x_{n})$ in a specific
way, we will find at most a fixed number of $\varepsilon$-fluctuations. 

However this is far from the only learning procedure we can set up
which serves to check whether a sequence converges (up to some error
term). An abstract version of this (which works just as well for any
monotone $\Sigma_{2}^{0}$ formula with a single sequence parameter,
not just Cauchy statements) is that we have a \emph{learning functional
}$L(j,(x_{n}))$, and we define our list of guesses $c_{1},c_{2},\ldots$
by $c_{0}=0$, and at the $j$th stage, put 
\[
c_{i+1}=L(j,(x_{n}))\mbox{ if }\neg\varphi_{0}(j,c_{i},(x_{n}))\wedge\forall j^{\prime}<j,\varphi_{0}(j^{\prime},c_{i},(x_{n}))
\]
and otherwise we keep the old $c_{i}$. Associated to the learning
procedure $L(j,(x_{n}))$, we have a \emph{bound} functional $B((x_{n}))$,
for which the following sentence holds: 
\[
\exists i\leq B((x_{n})),\forall j,\varphi_{0}(j,c_{i},(x_{n})).
\]
In other words, $B$ takes a sequence and gives us an upper bound
on how many ``mind changes'' $L$ needs to make (in other words,
how many candidates $c_{i}$ $L$ needs to come up with before it
finds one that works for all $j\in\mathbb{N}$ in the formula $\varphi_{0}(j,c_{i},(x_{n}))$). 

We say that a formula $\varphi$ (which is monotone $\Sigma_{2}^{0}$
with a single sequence parameter) is $(B,L)$\emph{-learnable} if
such a pair of functionals $B$ and $L$ exist. We say that $\varphi$
is \emph{effectively $(B,L)$}-learnable if moreover $B$ and $L$
are computable relative to $(x_{n})$. Then, specifying a family of
sequences which the parameter $(x_{n})$ is allowed to range over
corresponds to asking whether a specific property \emph{about} this
family of sequences is $(B,L$)-learnable. 

(Kohlenbach and Safarik then proceed to relax the monotonicity requirement,
but the more general definition of a learning procedure is more intricate.) 

The paper goes on to study the circumstances under which it is possible
to mechanically extract the functionals $B$ and $L$, and demonstrate
that effective learnability is a strictly stronger form of effective
convergence than metastability, but weaker than an effective bound
on fluctuations. (The paper also gives a sufficient condition for
when it is possible to \emph{deduce} an effective bound on fluctuations
from effective learnability.) 

To tie all of this abstraction back to the topic at hand, we briefly
illustrate how fluctuations at distance $\beta$ can be viewed as
a special type of $(B,L$)-learnability. (Consequently, it follows
that from a computability standpoint, fluctuations at distance $\beta$
is stronger than metastability.) Fix $\varepsilon=2^{-k}$, and suppose
that $\mathcal{S}$ is a family of sequences which has a uniform bound
$K$ on the number of $\varepsilon$-fluctuations at distance $\beta$.
On $\mathcal{S}$, we run the following learning procedure: initialize
$c_{0}=0$, and put 
\[
c_{i+1}=j\mbox{ if }\neg\varphi_{0}(j,c_{i},(x_{n}))\wedge\forall j^{\prime}<j,\varphi_{0}(j^{\prime},c_{i},(x_{n}))
\]
where $\varphi_{0}(j,n,(x_{n}))$ a modified form of the Cauchy sequence,
namely $(j\geq\beta(n)\rightarrow d(x_{j},x_{n})<2^{-k})$. This learning
procedure runs as follows: if an $\varepsilon$-fluctuation is detected
(specifically, $d(x_{j},x_{c_{i}})\geq2^{-k}$), then we put $j=c_{i+1}$.
We then \emph{ignore} all potential witnesses of $\varepsilon$-fluctuations
until we get to $\beta(c_{i+1})$, and then begin searching for a
new index $j^{\prime}$ such that $d(x_{j^{\prime}},x_{c_{i+1}})\geq2^{-k}$. 

The assumption that $\mathcal{S}$ has a uniform bound on $\varepsilon$-fluctuations
at distance $\beta$ implies that this learning procedure has at most
$K$ many mind-changes. Thus for $\mathcal{S}$, we can put $B((x_{n}))=K$
and $L(j,(x_{n}))=j$, and $\mathcal{S}$ is effectively learnable
(with respect to our modified version of the Cauchy sequence that
encodes the ``at distance $\beta$'' condition). 

\section*{B.2 Weak modes of uniform convergence and nonstandard compactness
principles}

This section is not really self-contained, but is being included to
indicate, for posterity, some of the genesis of the project described
in the main text. \emph{Caveat lector}. (The cited works are, however,
eminently readable, and the author would certainly encourage the curious
reader to peruse them.) 

Tacitly lurking beneath a great deal of this thesis is a nonstandard
analysis/ultraproduct interpretation of things like metastability
and bounds on fluctuations. In the main part of the text, we derived
our fluctuation bound, in some sense, with bare hands. But how might
one come to believe that it would be \emph{possible }to do so, or
what sort of data the bound would need to depend on?

Before we address those questions, we demonstrate the connection between
metastability and convergence properties ``in the ultraproduct''.
We quickly recap some nonstandard analysis facts. Here, $*$ denotes
the transfer functor (a.k.a. nontstandard extension/embedding); no
harm would come to the reader to think of the transfer functor as
taking an ultrapower (with an index set possibly much larger than
$\mathbb{N}$). We say that a nonstandard sequence (for instance,
an internal function from$^{*}\mathbb{N}$ to $^{*}\mathbb{R}$) is
\emph{externally Cauchy }if, for every \textbf{standard} $\varepsilon>0$
there exists a \textbf{standard} $N$ such that for all \textbf{standard
$n,m\geq N$, $|x_{n}-x_{m}|>\varepsilon$}. (Contrast with this with
\emph{internally Cauchy}, which corresponds to replacing all the bolded
``standard''s with ``standard or nonstandard'': the property of
being internally Cauchy comes from applying the $*$ functor to the
property of being Cauchy in the ordinary sense.) Externally/internally
metastable is defined in a similar fashion; see also the proof below. 
\begin{prop}
\emph{(Avigad-Iovino \cite{avigad2013ultraproducts}, Tao \cite{tao2012metastable})}
Let $\mathcal{C}$ be a class whose elements are pairs $\langle(X,d),(a_{n})\rangle$
of metric spaces $(X,d)$ together with distinguished sequences $(a_{n})$.
We can consider the class $^{*}\mathcal{C}$, whose elements are internal
metric spaces together with internal ($^{*}\mathbb{N}$-indexed) sequences.
TFAE:

\begin{enumerate}
\item The sequences $(a_{n})$ of $\mathcal{C}$ admit a uniform rate of
metastability, i.e. a function $\psi(F,\varepsilon)$ which jointly
witnesses the metastability of every $(a_{n})$. 
\item Every sequence in $^{*}\mathcal{C}$ is externally metastable. 
\item Every sequence in $^{*}\mathcal{C}$ is externally Cauchy. 
\end{enumerate}
\end{prop}

\begin{proof}
$(1\Rightarrow2)$ Let $\psi(F,\varepsilon)$ be a uniform rate of
metastability for $\mathcal{C}$. By transfer, for every internal
metric space in $^{*}\mathcal{C}$, $^{*}\psi(F,\varepsilon)$ is
a uniform bound on
{} nonstandard metastability, in the sense that for every distinguished
sequence $(a_{n})_{n\in^{*}\mathbb{N}}$, and every internal function
$F$ from $^{*}\mathbb{N}$ to $^{*}\mathbb{N}$, and every $\varepsilon\in\phantom{}^{*}\mathbb{R}_{+}$,
then $^{*}\psi(F,\varepsilon)$ returns an $N\in\phantom{}^{*}\mathbb{N}$
such that for all $m,n\in[N,F(N)]$, $^{*}d(a_{n},a_{m})<\varepsilon$.
By the principle that standard functions take standard values at standard
points, if $F:\phantom{}^{*}\mathbb{N}\rightarrow\phantom{}^{*}\mathbb{N}$
is the transfer of a standard function and $\varepsilon\in\mathbb{R}_{+}$,
then $\psi(F,\varepsilon)=N$ and $F(N)$ are both in $\mathbb{N}$.
Thus, $(a_{n})\upharpoonright\mathbb{N}$ is externally metastable.

( $2\Leftarrow1$ ) Conversely, assume that no uniform rate of metastability
exists for $\mathcal{C}$. Fix $F:\mathbb{N}\rightarrow\mathbb{N}$,
and fix $\varepsilon$, and suppose that there exists a sequence $(X_{i},d_{i})_{i\in\mathbb{N}}$
of metric spaces such that the least $N_{i}$s witnessing the metastability
of the sequences $(a_{n,i})_{i\in\mathbb{N}}$ form a divergent sequence
of natural numbers. By transfer, each $N_{i}$ is also the appropriate
witness of metastability for $^{*}(a_{n,i})$ on the internal metric
space $^{*}(X,d)$. Now consider the set of hypernaturals which are
less than some $^{*}F,\varepsilon$-witness of metastability for some
element of $^{*}\mathcal{C}$, in other words the set 
\[
\left\{ k\in\phantom{}^{*}\mathbb{N}\mid(\exists\{(X,d),(a_{n})\}\in\phantom{}^{*}\mathcal{C})\min\left\{ N\in\phantom{}^{*}\mathbb{N}\mid\forall n,m\in[N,\phantom{}^{*}F(N)],d(a_{n}a_{m})<\varepsilon]\right\} >k\right\} 
\]
which by assumption now contains every $n\in\mathbb{N}$. However
it is internally defined, so by overspill it contains some initial
segment of $\phantom{}^{*}\mathbb{N}\backslash\mathbb{N}$. In turn
this implies that there is some internal metric space in $^{*}\mathcal{C}$
such that the least $N$ witnessing the metastability for $^{*}F$
and $\varepsilon$, of the distinguished sequence $(a_{n})$, is in
$^{*}\mathbb{N}$, so in particular no \emph{standard} witness for
metastability at $^{*}F$ and $\varepsilon$ exists, therefore $(a_{n})$
is not externally metastable.

$(2\Leftrightarrow3)$ It remains to be shown that external metastability
is equivalent to the external Cauchy property. But while this does
not follow by transfer per se, the argument is identical to the standard
case. If $(a_{n})$ is externally Cauchy, then for every standard
$\varepsilon$ there exists a standard $N$ such that for all $n,m\geq N$,
$d(a_{n},a_{m})<\varepsilon$. A fortiori, for any increasing $F:\mathbb{N}\rightarrow\mathbb{N}$
and all $n,m\in[N,F(N)]$, $d(a_{n},a_{m})<\varepsilon$. Conversely,
if $(a_{n})$ is not externally Cauchy, there exists a standard $\varepsilon$
such that for all standard $N$, there exist standard $n,m\geq N$
such that $d(a_{n},a_{m})<\varepsilon$. Picking $F:\mathbb{N}\rightarrow\mathbb{N}$
such that $n,m\in[N,F(N)]$, we see that $d(a_{n},a_{m})$ is not
externally metastable. 
\end{proof}
\begin{rem*}
We present this proof in the ``synthetic'' style of nonstandard
analysis, in contrast with the proof given in Avigad-Iovino \cite{avigad2013ultraproducts},
which is carried out using explicit manipulation of an ultraproduct.
The distinction between these two approaches is largely one of taste;
the underlying idea of the proof is the same.

Analogous compactness theorems, relating types of nonstandard convergence
to uniform bouds on fluctuations, and fluctuations at distance $\beta$
more generally, have been recently given in Towsner's paper \emph{Nonstandard
analysis gives bounds on jumps }\cite{towsner2017nonstandard}.
\end{rem*}
What this proposition tells us is that one way to test for the existence
of a uniform bound on the rate of metastability for a class of sequences
is to take the ultraproduct/nonstandard extension, and check whether
an arbitrary ($^{*}\mathbb{N}$-indexed) sequence ``in the ultraproduct''
is externally Cauchy. 

This may sound like more work than it's worth! But in certain cases
it can be a handy diagnostic tool. Indeed, if the class $\mathcal{C}$
of sequences has some associated convergence proof (like say, a class
of ergodic averages), a natural way to test for external Cauchy convergence
is to check whether the \emph{same proof} applies to sequences in
the ultraproduct, mutatis mutandis. Typically, this amounts to requiring
that all of the classes of objects named in the proof are closed under
ultraproducts - for example, if the proof mentions reflexive Banach
spaces, it is known that reflexive Banach spaces are not closed under
ultraproducts, so if the original proof essentially relies on reflexivity,
we know this will be lost in the ultraproduct and so the original
proof does not ``pass through to the ultraproduct'', and (ultimately)
we are unable to prove external Cauchy convergence. 

To tie this back to some of the objects we saw earlier in the thesis:
it is known that uniformly convex Banach spaces with a specified modulus
of uniform convexity are closed under ultraproducts. Likewise, it
is known that amenable groups are \emph{not }closed under ultraproducts,
but there are several moduli one can affix to a class of amenable
groups that result in that class being closed under ultraproducts. 

To see this type of argument in action, we refer the reader to either
Avigad and Iovino's \emph{Ultraproducts and Metastability }\cite{avigad2013ultraproducts}
or Towsner's \emph{Nonstandard analysis gives bounds on jumps} \cite{towsner2017nonstandard}.
An (unpublished) argument of this type was used to by the author derive
an earlier prototype of the main result of this thesis, namely that
for countable amenable groups acting on Hilbert spaces, there exists
a uniform bound on the rate of metastability which depends on $\varepsilon$,
$\Vert f\Vert$, and some data from the choice of Folner sequence.
In fact, it was reading Towsner's paper that caused the author to
suspect that it might be possible at all to get some sort of fluctuation
bound in the amenable setting.

However, these ultraproduct compactness arguments only indicate that
a uniform bound on the rate of metastability (resp. number of fluctuations),
depending only on certain data, \emph{exists}, in a non-constructive
sense; it does not actually tell us anything about what the bound
looks like, or even that the bound is \emph{computable} from the stipulated
data. In practice, it is not hard to informally reverse-engineer the
ultraproduct argument to get an \emph{explicit }bound, as we have
essentially done; it is an interesting question whether/under what
circumstances this type of reverse engineering can be \emph{mechanized},
say in the form of a proof translation. 

\bibliographystyle{plain}
\nocite{*}
\bibliography{ms}

\end{document}